\documentclass[10pt,reqno]{amsart}

\usepackage{amsmath}
\usepackage{amsthm}
\usepackage{amsfonts}
\usepackage{amssymb}
\usepackage{graphicx}
\usepackage{bm}
\usepackage{multirow}%
\newcommand{\N}{\mathbb{N}}

\newcommand{\E}{\mathbb{E}}

\newtheorem{theorem}{Theorem}

\newtheorem{corollary}[theorem]{Corollary}

\newtheorem{definition}[theorem]{Definition}
\newtheorem{comments}[theorem]{Comments}
\newtheorem{example}[theorem]{Example}
\newtheorem{examples}[theorem]{Examples}

\newtheorem{lemma}[theorem]{Lemma}

\newtheorem{proposition}[theorem]{Proposition}
\newtheorem{remark}[theorem]{Remark}
\renewenvironment{proof}{\textit{Proof}}{${}\hfill\square$}

\def\~{\tilde}

\begin{document}

\title{ On Darboux Theorem for symplectic forms
on direct limits 
of symplectic Banach manifolds}
\author{Fernand Pelletier}
\date{}
\maketitle

\begin{abstract}
Given an ascending sequence of weak symplectic Banach manifolds on which the Darboux  theorem is true,
we can ask about conditions under which  the Darboux Theorem is also true on the direct limit. We will show  in general, 
without very strong conditions, the answer is negative. In particular we give an example   of an ascending
weak symplectic Banach manifolds on which the Darboux Theorem is  true but not on the direct limit. In a second part,
 we illustrate this discussion in the context of  an ascending sequences  of Sobolev manifolds of loops in  symplectic finite dimensional manifolds. 
This context gives rise to an example of direct limit of weak symplectic Banach manifolds on which the Darboux  theorem is true
 around any point.

\end{abstract}


\section{Introduction}

For any finite dimensional symplectic manifold, the Darboux Theorem asserts
that, around each point, there exists a (Darboux) chart in which the $2$-form
can be written as a constant one  {\it i.e.} the classical linear Darboux form. Such a
result can be proved by induction on the dimension of the manifold. However
using an idea of Moser for volume form on compact manifold (\cite{Mos}), the
Darboux theorem can be also proved by using an isotopy obtained by the local
flow of a time dependent vector field. Such a method is classically called  the
\textit{Moser's method} and works in many other frameworks.\newline

In the Banach context, it is well known that a symplectic form can be strong
or weak (see Definition \ref{D_WeaklyNonDegenerateBilinearForm}). The Darboux
Theorem was firstly proved for strong  symplectic  Banach manifolds   Weinstein (\cite{Wei}). But
Marsden (\cite{Ma2}) showed that the Darboux theorem fails for a weak
symplectic Banach manifold. However Bambusi \cite{Bam} found necessary and
sufficient conditions for the validity of Darboux theorem for a weak
symplectic Banach manifold (\textit{Darboux-Bambusi Theorem}). The proof of all
these versions of Darboux Theorem were all established by Moser's
method.\newline

In a wider context like Fr\'{e}chet or convenient manifolds, a symplectic form
is always weak. Unfortunately,  the  Moser  method does  not work in this framework
since the flow of vector filed does not exist in general. Recently, a new
approach of differential geometry in Fr\'{e}chet context was initiated and
developed by G. Galanis, C. T. J. Dodson, E. Vassiliou and their collaborators
in terms of projective limits of Banach manifolds (see \cite{DGV} for a
panorama of these results). In this situation, P. Kumar, in \cite{Ku1}, proves
a version of Darboux Theorem Moser
method  under very strong assumptions.\newline

The first part of this work is devoted to the problem of the validity of   Darboux Theorem
on  a direct limit of Banach manifolds.  More precisely given a countable
ascending sequence
\[
M_{1}\subset M_{2}\subset\cdots\subset M_{n}\subset M_{n+1}\subset\cdots
\]
of Banach manifolds, then $M=\bigcup\limits_{n\in\mathbb{N}^{\ast}}M_{n}$ is
the direct limit of this ascending sequence. Under the assumption of existence
of "chart limit" (\textit{cf.} Definition \ref{DLChartProperty}), we can
provide $M$ with a structure of convenient structure (not necessary
Hausdorff). If we consider a sequence $\left(  \omega_{n}\right)
_{n\in\mathbb{N}^{\ast}}$ of (strong or weak) symplectic form $\omega_{n}$ on
$M_{n}$ such that $\omega_{n}$ is the restriction of $\omega_{n+1}$
to $M_{n}$ (for all $n>0$), we obtain a symplectic form $\omega$ on the
convenient manifold $M$. Then the validity of a Darboux theorem on direct
limit of Banach manifolds can be studied. In this situation, we give simple
\textit{necessary and sufficient conditions for the existence of a Darboux
chart around some point of the direct limit }(Theorem \ref{equivDarbouxascending}).
Unfortunately,  these conditions are not satisfied in general. For instance, we
can look for applying the Moser method on the direct limit of Banach symplectic
manifolds which satisfy the Darboux-Bambusi Theorem assumptions. But as in the
Fr\'{e}chet framework, \textit{without trivial cases or very strong
assumptions, a direct limit }$X$\textit{ of a sequence }$\left(  X_{n}\right)
_{n\in\mathbb{N}^{\ast}}$\textit{ of vector fields }$X_{n}$\textit{ on }%
$M_{n}$\textit{ does not have a local flow in general} (\textit{cf.} Appendix \ref{solODE}). 
This implies that the \textit{Moser's method does not work in
 general in this context}. This is explained in details in section
\ref{no_Darboux_ascending}. However, under \textit{very strong assumptions, we
could  prove a Darboux Theorem}. But it seems that these conditions are so restrictive that
 they are not satisfied in a large context.    Thus we believe that such a result would not be relevant.

To complete this approach, we produce \textit{an example} of direct limit of symplectic
Banach manifolds  $M_n$ on which \textit{there exists no Darboux chart} around some
point. In fact  around such a point $x_0$, we have a  Darboux chart $(U_n,\psi_n)$ in  $M_n$ for each $n\geq n_0$ 
but  we have  $
 \bigcap_{n\geq n_0}U_n=\{x_0\}$.
In the general context it is such a situation which is the essential reason for which the Darboux Theorem can be not true on  a 
 direct limit of symplectic Banach manifold ({\it cf.} Theorem \ref{equivDarbouxascending}).
\newline

As an illustration of this previous situation,  the second part of this paper is devoted to the
existence of a Darboux chart on a Sobolev manifold of loops in a finite
dimension symplectic manifold and  on a direct limit of such Banach manifolds. In
some words, we consider the Sobolev manifold $\mathsf{L}_{k}^{p}%
(\mathbb{S}^{1},M)$ of loops $\gamma:\mathbb{S}^{1}\rightarrow M$ of Sobolev
class $\mathsf{L}_{k}^{p}$, where $M$ is a manifold of dimension $m$. These
manifolds are modelled on the reflexive Banach spaces $\mathsf{L}_{k}%
^{p}(\mathbb{S}^{1},\mathbb{R}^{m})$. Now, if $\omega$ is a symplectic form on
$M$, we can define a natural symplectic form $\Omega$ on $\mathsf{L}_{k}%
^{p}(\mathbb{S}^{1},M)$ which satisfies the assumptions of Darboux-Bambusi
Theorem. Then, around each $\gamma\in\mathsf{L}_{k}^{p}(\mathbb{S}^{1},M)$, we
have a Darboux chart. Note that this situation \textit{gives an illustration
to Darboux-Bambusi Theorem which is new to our knowledge}. The problem of
existence of a Darboux chart on an ascending sequence of symplectic Sobolev
manifolds of loops is similar to the general context. However, there exists 
ascending sequences of such manifolds on which a Darboux theorem is true.
\newline

This work is self contained.

\noindent In section 2,  after a survey on  known results on symplectic forms on a Banach space,
we look for the context of Darboux-Bambusi Theorem.
In fact we recover this Theorem as a consequence of generalized version of "Moser's Lemma".

The context of direct limit of ascending sequence of symplectic Banach
manifolds is described in section 3. We begin with the case of a direct limit
of an ascending sequence of Banach spaces and we describe the link between
symplectic forms on the direct limit and a sequence of "coherent " symplectic
forms on each Banach space. 
Then we look for the same situation on an ascending sequence of Banach vector
bundles. We end the first part by a discussion on the problem the existence of a
Darboux Theorem on a direct limit of symplectic Banach manifolds, in the
general framework at first, and then under the assumption of Darboux-Bambusi
Theorem. In particular, in this situation we produce an example for which the
Darboux Theorem fails.

Section \ref{symplecticloops} essentially describes a symplectic structure
on Sobolev manifolds of loops. More precisely, after a survey of classical
results on Sobolev spaces, we recall how to define the Sobolev manifolds $\mathsf{L}%
_{k}^{p}(\mathbb{S}^{1},M)$ of loops $\gamma:\mathbb{S}^{1}\rightarrow M$ of
Sobolev class $\mathsf{L}_{k}^{p}$ in a manifold $M$ of dimension $m$. 
Then, when $(M,\omega)$ is a symplectic form, as in \cite{Ku2},
$\mathsf{L}_{k}^{p}(\mathbb{S}^{1},M)$ can be endowed with a natural
weak symplectic form $\Omega$. After, we show that the assumptions
of Darboux-Bambusi Theorem are satisfied for $\Omega$ on $\mathsf{L}_{k}%
^{p}(\mathbb{S}^{1},M)$. We also prove some complementary results using
Moser's method. Any direct limit of Sobolev manifolds of loops in an ascending
sequence symplectic manifold $\left\{  (M_{n},\omega_{n})\right\}
_{n\in\mathbb{N}^{\ast}}$ can be provided with a symplectic form which is
obtained (in the same way as previously) from the symplectic form $\omega$ on
the direct limit of manifolds $\left(  M_{n}\right)  $ (defined from $\left(
\omega_{n}\right)  $). We end this section with an example of an ascending
sequence of such manifolds on which a Darboux theorem is true at any point.
\newline

Finally, in Appendix A, we recall all  results  on
direct limit of manifolds and Banach bundles needed in this work. In Appendix B, we discuss about
the problem of existence of solutions of a direct limit of ODE on a direct
limit of Banach spaces. After having given strong sufficient conditions for
the existence of such solutions, we comment the coherence and the pertinence
of these conditions through examples and counter-examples


\section{Symplectic forms on a Banach manifold and Darboux Theorem}

\label{_SymplecticFormsOnBanachBundles}

\subsection{Symplectic forms on Banach space}

\label{linearsymplectic}

\begin{definition}
\label{D_WeaklyNonDegenerateBilinearForm}Let $\mathbb{E}$ be a Banach space. A
bilinear form $\omega$ is said to be weakly non degenerate if
$\left(  \forall Y\in\mathbb{E},\ \omega\left(  X,Y\right)  =0\right)
\quad\Longrightarrow\quad X=0$.

\end{definition}

Classically, to $\omega$ is associated the linear map 

$\;\;\;\; \omega^{\flat}:  \mathbb{E}  \longrightarrow  \mathbb{E}^{\ast}$
\noindent defined by $ \left(\omega^{\flat}(X)\right)(Y)=\omega\left(  X,Y\right),\;:\;  \forall Y\in \mathbb{E}$.

Clearly,  $\omega$ is weakly non degenerate if and only if  $\omega^{\flat}$ is
injective. \\
 The $2$-form $\omega$ is  called  \textit{strongly nondegenerate} if $\omega^\flat$ is an isomorphism.
.

\bigskip

A fundamental result in finite dimensional linear symplectic space is the
existence of a \textit{Darboux (linear) form} for a symplectic $2$-form:

If $\omega$ is a symplectic form on a finite dimensional vector space
$\mathbb{E}$, there exists a vector space $\mathbb{L}$ and an isomorphism
$A:\mathbb{E}\rightarrow\mathbb{L}\oplus\mathbb{L}^{\ast}$ such that
$\omega=A^{\ast}\omega_{_{\mathbb{L}}}$ where
\begin{equation}
\omega_{_{\mathbb{L}}}\left((u,\eta),(v,\xi)\right)=<\eta,v>-<\xi,u> \label{Darboux}%
\end{equation}
This result is in direct relation with the notion of \textit{Lagangian
subspace} which is a fundamental tool in the finite dimensional symplectic framework.

\bigskip

In the Banach framework, let $\omega$ be a weak symplectic form on a Banach space.

A subspace $\mathbb{F}$ is \textit{isotropic} if $\omega(u,v)=0$ for all
$u,v\in\mathbb{F}$. An isotropic subspace is always closed. \newline If
$\mathbb{F}^{\perp_{\omega}}=\{w\in\mathbb{E}\;:\forall u\in\mathbb{F}%
,\ \omega(u,v)=0\;\}$ is the orthogonal symplectic space of $\mathbb{F}$, then
$\mathbb{F}$ is isotropic if and only if $\mathbb{F}\subset\mathbb{F}%
^{\perp_{\omega}}$ and is \textit{maximal isotropic if }$\mathbb{F}%
=\mathbb{F}^{\perp_{\omega}}$\textit{.} Unfortunately, in the Banach
framework, a maximal isotropic subspace $\mathbb{L}$ can be not supplemented.
Following Weinstein's terminology (\cite{Wei}), an isotropic space
$\mathbb{L}$ is called a \textit{Lagrangian space} if there exists an
isotropic space $\mathbb{L}^{\prime}$ such that $\mathbb{E}=\mathbb{L}%
\oplus\mathbb{L}^{\prime}$. Since $\omega$ is strong non degenerate, this
implies that $\mathbb{L}$ and $\mathbb{L}^{\prime}$ are maximal isotropic and
then are Lagrangian spaces (see \cite{Wei}).

Unfortunately, in general, for a given symplectic structure, Lagrangian
subspaces need not exist (\textit{cf. }\cite{KaSw}). Even for a strong
symplectic structure on Banach space which is not Hilbertizable, the non
existence of Lagrangian subspaces is an open problem to our knowledge.
Following \cite{Wei}, a symplectic form $\omega$ on a Banach space
$\mathbb{E}$ is a \textit{Darboux (linear) form} if there exists a Banach
space $\mathbb{L}$ and an isomorphism $A:\mathbb{E}\rightarrow\mathbb{L}%
\oplus\mathbb{L}^{\ast}$ such that $\omega=A^{\ast}\omega_{\mathbb{L}}$ where
$\omega_{\mathbb{L}}$ is defined in (\ref{Darboux}). Note that in this case
$\mathbb{E}$\textit{ must be reflexive}.

\begin{examples}\normalfont${}$
\begin{enumerate}
\item[1.] If $\omega$ is a linear Darboux form on a Hilbert space $\mathbb{H}%
$, it is always true that there exists a Lagrangian decomposition and all
symplectic forms are isomorphic to such a Darboux form (\textit{cf.}
\cite{Wei}). However, for a general Banach space, this is not true as the
following example shows (cf. \cite{KaSw}).\newline${}\;\;$ It is well known
that for $1\leq p<\infty$, the vector space $\ell^{p}$ of $p$-summable real
sequences is a Banach space which can be written as $l^{p}=l_{1}^{p}\oplus
l_{2}^{p}$, where $l_{1}^{p}$ and $l_{2}^{p}$ are Banach subspaces which are
isomorphic to $l^{p}$. Now, for $p\not =2$, consider $\mathbb{E}=l^{p}\oplus
l^{q}$, where $l^{q}$ is isomorphic to $(l^{p})^{\ast}$. For the canonical
Darboux form on $\mathbb{E}$, we have at least two types of non "equivalent
Lagrangian decompositions": one is $\mathbb{E}=l^{p}\oplus l^{q}$ for which
$l^{p}$ is not isomorphic to $l^{q}$ and another one is $\mathbb{E}=(l_{1}%
^{p}\oplus l_{1}^{q})\oplus(l_{2}^{p}\oplus l_{2}^{q})$ for which the
Lagrangian subspace $(l_{1}^{p}\oplus l_{1}^{q})$ is isomorphic to the
Lagrangian subspace $(l_{2}^{p}\oplus l_{2}^{q})$.

\item[2.] Given a Hilbert space $\mathbb{H}$ endowed with an inner product
$<\;,\;>$; thanks to Riesz theorem, we can identify $\mathbb{H}$ with
$\mathbb{H}^{\ast}$. Therefore, we can provide $\mathbb{H}\times\mathbb{H}$
with a canonical Darboux symplectic form $\omega$. Consider now any Banach
space $\mathbb{E}$ which can be continuously and densely embedded in
$\mathbb{H}$, then $\omega$ induces a symplectic form on $\mathbb{E}%
\times\mathbb{E}$ in an obvious way. This situation occurs for instance with
$\mathbb{H}=l^{2}(\mathbb{N})$ and $\mathbb{E}=l^{1}(\mathbb{N})$.

\item[3.] Let $\widehat{\omega}$ be a Darboux form on a Banach space
$\widehat{\mathbb{E}}$ and let $A:\mathbb{E}\rightarrow\widehat{\mathbb{E}}$
be an injective continuous linear map. We set $\omega=A^{\ast}\widehat{\omega
}$. On $\widehat{\mathbb{E}}$ we have a decomposition $\widehat{\mathbb{E}%
}=\widehat{\mathbb{L}}\oplus\widehat{\mathbb{L}}^{\ast}$. Set $\mathbb{L}%
=A^{-1}(\widehat{\mathbb{L}})$ and $\mathbb{L}^{\prime}=A^{-1}(\widehat
{\mathbb{L}})$, then we have $\mathbb{E}=\mathbb{L}\oplus\mathbb{L}^{\prime}$.
Moreover, since $\widehat{\mathbb{L}}$ and $\widehat{\mathbb{L}}^{\ast}$ are
Lagrangian (relatively to $\widehat{\omega}$) so $\mathbb{L}$ and
$\mathbb{L}^{\prime}$ are Lagrangian (relatively to $\omega$).\newline Using
analogous arguments as in Point 1., if $p\not =2$, this situation occurs for
$\hat{E}=L^{p}([0,1])\oplus L^{q}([0,1])$ (where $L^{q}([0,1])$ is isomorphic
to $(L^{p}([0,1]))^{\ast}$) and $\mathbb{E}=L^{r}([0,1])\oplus L^{r}([0,1])$
for $r\geq\sup\{p,q\}$ and $A$ is the canonical embedding of $L^{r}%
([0,1])\oplus L^{r}([0,1])$ into $L^{p}([0,1])\oplus L^{q}([0,1])$. Note that,
in this case, $\hat{\mathbb{E}}$ is a reflexive Banach space which is not Hilbertizable.
\end{enumerate}
\end{examples}

\begin{remark}\normalfont
\label{nondecomposable}Note that for all the previous examples of symplectic
forms, we have a decomposition of the Banach spaces into Lagrangian spaces and 
the symplectic forms are some pull-backs of a Darboux form. The reader can
find in \cite{Swa} an example of a symplectic form on a Banach space which is
not the pull-back of a Darboux form. \newline
\end{remark}

\bigskip

Let $\mathbb{E}$ be a Banach space provided with a norm $||\;||$. We consider
a symplectic form $\omega$ on $\mathbb{E}$ and let $\omega^{\flat}%
:\mathbb{E}\rightarrow\mathbb{E}^{\ast}$ be the associated bounded linear
operator. Following \cite{Bam} and \cite{Ku1}, on $\mathbb{E}$, we consider
the norm $||u||_{\omega}=||\omega^{\flat}(u)||^{\ast}$ where $||\;||^{\ast}$
is the canonical norm on $\mathbb{E}^{\ast}$ associated to $||\;||$. Of
course, we have $||u||_{\omega}\leq||\omega^{\flat}||^{\mathrm{op}}.||u||$
(where $||\omega^{\flat}||^{\mathrm{op}}$ is the norm of the operator
$\omega^{\flat}$) and so the inclusion of the normed space $(\mathbb{E}%
,||\;||)$ in $(\mathbb{E},||\;||_{\omega})$ is continuous. We denote by
$\widehat{\mathbb{E}}$ the Banach space which is the completion of
$(\mathbb{E},||\;||_{\omega})$. Since $\omega^{\flat}$ is an isometry from
$(\mathbb{E},||\;||_{\omega})$ onto its range in  $\mathbb{E}^{\ast}$, we can extend
$\omega^{\flat}$ to a bounded operator $\widehat{\omega}^{\flat}$ from
$\widehat{\mathbb{E}}$ to $\mathbb{E}^{\ast}$. Assume that $\mathbb{E}$ is
\textit{reflexive}. Therefore $\hat{\omega}^{\flat}$ is an isometry between
$\widehat{\mathbb{E}}$ and $\mathbb{E}^{\ast}$ (\cite{Bam} Lemma 2.7).
Moreover, $\omega^{\flat}$ can be seen as a bounded linear operator from
$\mathbb{E}$ to $\widehat{\mathbb{E}}^{\ast}$ and is in fact an isomorphism
(\cite{Bam} Lemma 2.8). Note that since $\widehat{\mathbb{E}}^{\ast}$
\textit{is reflexive}, this implies that $\hat{\mathbb{E}}$ \textit{is also
reflexive.} 

\begin{remark}\label{indnorm} \normalfont If ${||\;||'}$ is an equivalent norm of $||\;||$  on $\E$, then the corresponding $(||\;||')^*$ and $||\;||^*$ are also equivalent norm on $\E^*$ and so ${||\;||'}_\omega$ and $||\;||_\omega$ are equivalent norms on $\E$  and then the completion $\widehat{\E}$   depends only of the  Banach structure on $\E$  defined by equivalent the norms on $\E$
\end{remark}


\subsection{A Moser's Lemma and a Darboux Theorem on a Banach
manifold\label{DarbouxBanach}}


In this section, we will prove a generalization of Moser's Lemma (see for
instance \cite{Les}) to the Banach framework, and, as a corollary, we obtain
the result of \cite{Bam} for a weak symplectic Banach manifold. \newline

A \textit{weak symplectic form} on a Banach manifold modelled on a Banach
space $\mathbb{M}$ is a closed $2$-form $\omega$ on $M$, which is non-degenerate.

Then $\omega^{\flat}:TM\rightarrow T^{\ast}M$ is an injective bundle morphism.
The symplectic form $\omega$ is weak if $\omega^{\flat}$ is not surjective.
\textit{Assume that }$\mathbb{M}$\textit{ is reflexive}. According to the end of  section \ref{linearsymplectic} or
 \cite{Bam}, we denote by $\widehat{T_{x}M}$ the
Banach space which is the completion of $T_{x}M$ provided with the norm
$||\;||_{\omega_{x}}$ associated to some choice norm $||\;||$ on $T_xM$ and the Banach space
 $\widehat{T_{x}M}$ does not depends of this choice (see Remark \ref{indnorm}) .
Then $\omega_{x}$ can be extended to a
continuous bilinear map $\hat{\omega}_{x}$ on $T_{x}M\times\widehat{T_{x}M}$
and $\omega_{x}^{\flat}$ becomes an isomorphism from $T_{x}M$ to
$(\widehat{T_{x}M})^{\ast}$. We set
\[
\widehat{TM}=\bigcup_{x\in M}\widehat{T_{x}M}\;\text{ and }\;(\widehat
{TM})^{\ast}=\bigcup_{x\in M}(\widehat{T_{x}M})^{\ast}.
\]

\begin{theorem}
[Moser's Lemma]\label{Moserlemma} Let $\omega$ be a weak symplectic form on a
Banach manifold $M$ modelled on a reflexive Banach space $\mathbb{M}$. Assume
that we have the following properties:

\begin{description}
\item[(i)] There exists a neighbourhood $U$ of $x_{0}\in M$ such that
$\widehat{TM}_{|U}$ is a trivial Banach bundle whose typical fiber is the
Banach space $(\widehat{T_{x_{0}}M},||\;||_{\omega_{x_{0}}})$;

\item[(ii)] $\;\omega$ can be extended to a smooth field of continuous
bilinear forms on \newline$TM_{|U}\times\widehat{TM}_{|U}$.
\end{description}

\noindent Consider a family $\{\omega^{t}\}_{0\leq t\leq1}$ of closed
$2$-forms which smoothly depends on $t$ with the following properties:

\begin{description}
\item $\omega^{0}=\omega$ and, for all $0\leq t\leq1$, $\omega_{x_{0}}^{t}
=\omega_{x_{0}}$;

\item $\;\omega^{t}$ can be extended to a smooth field of continuous bilinear
forms on $TM_{|U}\times\widehat{TM}_{|U}$.
\end{description}

Then there exists a neighbourhood $V$ of $x_{0}$ such that each $\omega^{t}$
is a symplectic form on $V$, for all $0\leq t\leq1$, and there exists a family
$\{F_{t}\}_{0\leq t\leq1}$ of diffeomorphisms $F_{t}$ from a neighbourhood
$V_{0}\subset V$ of $x_{0}$ to a neighbourhood  $F_{t}(V_{0})\subset V$  of $x_0$ such
that $F_{0}=Id$ and $F_{t}^{\ast}\omega^{t}=\omega$, for all $0\leq t\leq1$.
\end{theorem}

\begin{proof}
Let $(U,\phi)$ be a chart around $x_{0}$ such that $\phi(x_{0})=0\in
\mathbb{M}$ and consider that all the assumptions of the theorem are true on
$U$. According to the notations of section \ref{S_linearSymplectic}, let
$\Psi:\widehat{TM}_{|U}\rightarrow\phi(U)\times\widehat{\mathbb{M}}$ be a
trivialization. Without loss of generality, we may assume that $U$ is an open
neighbourhood of $0$ in $\mathbb{M}$ and $\widehat{TM}_{|U}=U\times
\widehat{\mathbb{M}}$. Therefore, $U\times\widehat{\mathbb{M}}$ is a trivial
Banach bundle modelled on the Banach space $(\widehat{\mathbb{M}%
},||\;||_{\omega_{0}})$.
Since $\omega$ can be extended to a non-degenerate skew symmetric bilinear
form (again denoted $\omega$) on $U\times(\mathbb{M}\times\widehat{\mathbb{M}%
})$ then $\omega^{\flat}$ is a Banach bundle isomorphism from $U\times
\mathbb{M}$ to $U\times\widehat{\mathbb{M}}^{\ast}$.\newline We set
$\dot{\omega}^{t}=\frac{d}{dt}\omega^{t}$. Since each $\omega^{t}$ is closed
for $0\leq t\leq1$, we have :
\[
d\dot{\omega}^{t}=\frac{d}{dt}(d\omega^{t})=0
\]
and so $\dot{\omega}^{t}$ is closed. After restricting U if necessary, from
the Poincar\'{e} Lemma, there exists a $1$-form $\alpha^{t}$ on $U$ such that
$\dot{\omega}^{t}=d\alpha^{t}$ for all $0\leq t\leq1$. In fact $\alpha_{t}$
can be given by
\[
\alpha_{x}^{t}=\int_{0}^{1}s.(\dot{\omega}^{t})_{sx}^{\flat}(x)ds
\]
Now, at $x=0$, $(\omega_{x_{0}}^{t})^{\flat}$ is an isomorphism from
$\mathbb{M}$ to $\widehat{\mathbb{M}}^{\ast}$. Since, for all $0\leq t\leq1$,
$x\mapsto(\omega_{x}^{t})^{\flat}$ is a smooth field from $U$ to
$\mathcal{L}(\mathbb{M},\widehat{\mathbb{M}}^{\ast})$, there exists a
neighbourhood $V$ of $0$ such that $(\omega_{x}^{t})^{\flat}$ is an
isomorphism from $\mathbb{M}$ to $\widehat{\mathbb{M}}^{\ast}$ for all $x\in
V$ and $0\leq t\leq1$. In particular, $\omega^{t}$ is a symplectic form on
$V$. Moreover $x\mapsto(\dot{\omega}_{x}^{t})^{\flat}$ is smooth and takes
values in $\mathcal{L}(\mathbb{M},\widehat{\mathbb{M}}^{\ast})$. Therefore
$x\mapsto\alpha_{x}^{t}$ can be extended to a smooth field on $V$ into
$\widehat{\mathbb{M}}^{\ast}$. We set $X_{x}^{t}:=-((\omega_{x}^{t})^{\flat
})^{-1}(\alpha_{x}^{t})$. It is a well defined time dependent vector field and
let $F_{t}$ be the flow generated by $X^{t}$ defined on some neighbourhood
$V_{0}\subset V$ of $0$. Note for all $t\in\lbrack0,1]$ since $\dot{\omega
}_{x_{0}}^{t}=0$ then $X_{x_{0}}^{t}=0$. Thus, for all $t\in\lbrack0,1]$,
$F_{t}(x_{0})=x_{0}$ . As classically, we have
\[
\frac{d}{dt}F_{t}^{\ast}\omega^{t}=F_{t}^{\ast}(L_{X^{t}}\omega^{t}%
)+F_{t}^{\ast}\frac{d}{dt}\omega^{t}=F_{t}^{\ast}(-d\alpha^{t}+\dot{\omega
}^{t})=0
\]
Thus $F_{t}^{\ast}\omega^{t}=\omega$.\newline${}\hfill$
\end{proof}

\begin{remark}
\label{strongsymplectic}  \normalfont If $\omega$ is a strong symplectic form on $M$, then
$\mathbb{M}$ is reflexive and $\omega^{\flat}$ is a bundle isomorphism from
$TM$ to $T^{\ast}M$. In particular, the norm $||\;||_{\omega_{x}}$ is
equivalent to any norm $||\;||$ on $T_{x}M$ which defines its Banach structure
and so all the assumptions (i) and (ii) of Theorem \ref{Moserlemma} are
locally always satisfied. Moreover, if we consider a family $\{\omega
^{t}\}_{0\leq t\leq1}$ as in Theorem \ref{Moserlemma} , the assumption
"$\;\omega^{t}$ can be extended to a smooth field of continuous bilinear forms
on $TM_{|U}\times\widehat{TM}_{|U}$ for all $0\leq t\leq1$" is always satisfied.

Therefore we get the same conclusion only with the assumptions $\omega
^{0}=\omega$ and $\omega_{x_{0}}^{t}=\omega_{x_{0}}$ for all $0\leq t\leq1$.
\newline
\end{remark}

\begin{example}\normalfont
We consider a symplectic form $\omega$ on $M$ such that the assumptions (i)
and (ii) of Theorem \ref{Moserlemma} are satisfied on some neighbourhood $U$
of $x_{0}$. We consider the family $\omega_{x}^{s,t}=e^{st}\omega_{x}$ where
$\left(  s,t\right)  \in\lbrack0,1]^{2}$. Then for $s=0$, we have
$\omega_{x_{0}}^{s,t}=\omega_{x_{0}}$ for all $t\in\lbrack0,1]$ and
$\omega^{0}=\omega$ for $t=0$. But $(\omega_{x_{0}}^{s,t})^{\flat}\equiv
\omega_{x_{0}}^{\flat}$ belongs to $\mathcal{L}(\mathbb{M},\widehat
{\mathbb{M}}^{\ast})$. Since the map $%
\begin{array}
[c]{ccc}%
\lbrack0,1]^{2}\times U & \longrightarrow & \mathcal{L}(\mathbb{M}%
,\widehat{\mathbb{M}}^{\ast})\\
(s,t,x) & \mapsto & (\omega_{x}^{s,t})^{\flat}%
\end{array}
$is smooth, there exists an open neighbourhood $\Sigma\times V\subset
\lbrack0,1]^{2}\times U$ of $\{0\}\times\lbrack0,1]\times\{x_{0}\}$ such that
$(\omega_{x}^{s,t})^{\flat}$ belongs to $\operatorname*{GL}(\mathbb{M}%
,\widehat{\mathbb{M}}^{\ast})$. Therefore the assumptions of Theorem
\ref{Moserlemma} are fulfilled for each $s$ fixed such that $(s,t)\in\Sigma$
on an open neighbourhood $V_{s}$ of $x_{0}$. This implies that, for any fixed
$s$ such that $\{s\}\times\lbrack0,1]\subset\Sigma$, there exists a family
$\{F_{s,t}\}_{0\leq t\leq1}$ of diffeomorphisms $F_{s,t}$ from a neighbourhood
$W_{s}$ of $x_{0}$ to $F_{s,t}(W_{s})\subset V_{s}$ is a neighouhood $x_{0}$
and such that $F_{0}=Id$ and $F_{s,t}^{\ast}\omega^{s,t}=\omega$ for all
$0\leq t\leq1$. \newline
\end{example}

Now as a Corollary of Theorem \ref{Moserlemma}, we obtain the Bambusi's
version of Darboux Theorem ( \cite{Bam} Theorem $2.1)$:\newline

\begin{theorem}
[Local Darboux Theorem]\label{localDarboux} Let $\omega$ be a weak symplectic
form on a Banach manifold $M$ modelled on a reflexive Banach space
$\mathbb{M}$. Assume that the assumptions (i) and (ii) of Theorem
\ref{Moserlemma} are satisfied. Then there exists a chart $(V,F)$ around
$x_{0}$ such that $F^{\ast}\omega_{0}=\omega$ where $\omega_{0}$ is the
constant form on $F(V)$ defined by $(F^{-1})^{\ast}\omega_{x_{0}}$.
\end{theorem}

\begin{proof}
Take a local chart $(U,\phi)$ around $x$ such that $\phi(x)=0$. On
$\phi(U)\subset\mathbb{M}$, we consider the symplectic form $\tilde{\omega
}=(\phi^{-1})^{\ast}\omega$. We set $\tilde{\omega}_{0}$ the constant $2$-form
on $\phi(U)$ given by $\tilde{\omega}_{0}$ and we consider the family
$\Omega^{t}=\tilde{\omega}_{0}+t(\tilde{\omega}-\tilde{\omega}_{0})$. Then the
reader can easily verify that $\omega^{t}=\phi^{\ast}\Omega^{t}$ satisfies the
assumptions of Theorem \ref{Moserlemma}.\newline${}\hfill$
\end{proof}

\begin{definition}
The chart $(V,F)$ in Theorem \ref{localDarboux} will be called a Darboux chart
around $x_{0}$.
\end{definition}

\begin{remark}
\label{comparebambusi}\normalfont
The assumptions of Theorem 2.1 in \cite{Bam} are  formulated in another way.
In fact, in this
Theorem 2.1,  the assumptions on all the norms $||\;||_{\omega_{x}}$  on the typical fiber 
$\widehat{\mathbb{M}}$ is a consequence of
the assumptions of Theorem \ref{Moserlemma}. Indeed, for a fixed $x_{0}$ in
$M$, there exists a trivialization $T\phi:TM_{|U}\rightarrow U\times
\mathbb{M}$. Therefore, there exists an isomorphism $\Psi_{x}$ of $\mathbb{M}$
such that $T_{x}\phi=T_{x_{0}}\phi\circ\Psi_{x}$ with $\Psi_{x_{0}}=Id$. Now
we have a trivialization $T\phi^{\ast}:U\times\mathbb{M}^{\ast}\rightarrow
T^{\ast}M_{|U}$ and we have $T\phi_{x}^{\ast}=\Psi_{x}^{\ast}\circ T_{x_{0}%
}\phi^{\ast}$; then $\Psi_{x}^{\ast}$ is an isomorphism of $\mathbb{M}^{\ast}%
$. After shrinking $U$ if necessary, there exists $K>0$ such that $\dfrac
{1}{K}\leq||\Psi_{x}||^{\mathrm{op}}\leq K$. It follows that $\alpha\mapsto||\Psi
_{x}(\alpha)||^{\ast}$ is a norm on $\mathbb{M}^{\ast}$ equivalent to
$||\;||^{\ast}$ on $\mathbb{M}^{\ast}$.

\end{remark}



\bigskip Weinstein presents an example of a weak symplectic form $\omega$ on a
neighbourhood of $0$ of a Hilbert $\mathbb{H}$ space for which the Darboux
Theorem is not true. The essential reason is that the operator $\omega^{\flat
}$ is an isomorphism from $T_{x}U$ onto $T_{x}^{\ast}U$ on $U\setminus\{0\}$
but $\omega_{0}^{\flat}$ is not surjective (for more details, see Example
\ref{nochartDarboux}). In this way, we have the following Corollary:

\begin{corollary}
\label{equivDarboux} Let $\omega$ be a weak symplectic form on a Banach
manifold modelled on a reflexive Banach space $\mathbb{M}$. Then there exists
a Darboux chart $(V,F)$ around $x_{0}$ if and only

\begin{description}
\item[(i)] there exists a chart $(U,\phi)$ around $x_{0}$ such that
$\widehat{TM}_{|U}$ is a trivial Banach bundle whose typical fiber is the
completion $\widehat{\mathbb{M}}$ of the normed space $(\mathbb{M}%
,||\;||_{(\phi^{-1})^{\ast}\omega_{x_{0}}})$ and $T\phi$ can be extended to a
trivialization $\widehat{T\phi}$~of $\widehat{TM}_{|U}$ on $U\times
\widehat{\mathbb{M}}$;

\item[(ii)] $\;\omega$ can be extended to a smooth field of continuous
bilinear forms on $TM_{|U}\times\widehat{TM}_{|U}$.
\end{description}
\end{corollary}

\begin{proof}
If we have a Darboux chart $(V,F)$ around $x_{0}$, then from Proposition 2.6
of \cite{Bam}, (i) is true. Since we have $\omega_{x}(X,Y)=\omega_{0}%
(T_{x}F(X),T_{x}F(Y))$ we have $\omega^{\flat}=T^{\ast}F\circ\omega_{0}%
^{\flat}\circ TF$.\newline Now, the inclusion $\iota:\mathbb{M}\rightarrow
\widehat{\mathbb{M}}$ is continuous with a dense range. Therefore we have an
injective bundle morphism, again denoted $\iota$, from $TM_{|U}$ to
$\widehat{TM}_{|U}$ whose range is dense. Since we can extend $TF$ to a
trivialization $\widehat{TF}$ of $\widehat{TM}_{|U}$ to $U\times
\widehat{\mathbb{M}}$. We can define $\widehat{\omega}^{\flat}=T^{\ast}%
F\circ\omega_{0}\circ\widehat{TF}$ which gives rise to a smooth field
$x\mapsto\widehat{\omega}_{x}$ of $2$-forms on $TM_{|U}\times\widehat{TM}%
_{|U}$. \newline Conversely, assume that there exists a chart $(U,\phi)$
around $x_{0}$ such that $T\phi$ can be extended to a $\widehat{T\phi}$ to
$\widehat{TM}_{|U}$ on $U\times\widehat{\mathbb{M}}$ where $\widehat
{\mathbb{M}}$ is the completion of the normed space $(\mathbb{M}%
,||\;||_{(\phi^{-1})^{\ast}\omega_{x_{0}}})$. Then the condition (i) of
Theorem \ref{Moserlemma} is satisfied. Since condition (ii) is assumed, the
result is a consequence of Theorem \ref{localDarboux}.\newline${}\hfill$
\end{proof}

Finally from Theorem \ref{localDarboux}, we also obtain the following global
version of a Darboux Theorem:\newline

\begin{theorem}
[Global Darboux Theorem]\label{globalDarboux} Let $\omega$ be a weak
symplectic form on a Banach manifold modelled on a reflexive Banach space
$\mathbb{M}$. Assume that we have the following assumptions:

\begin{description}
\item[(i)] $\widehat{TM}\rightarrow M$ is a Banach bundle whose typical fiber
is $\widehat{\mathbb{M}}$;

\item[(ii)] $\;\omega$ can be extended to a smooth field of continuous
bilinear forms on $TM\times\widehat{TM}_{|U}$.
\end{description}

Then for any $x_{0}\in M$, there exists a a Darboux chart $(V, F)$ around
$x_{0}$. \newline
\end{theorem}

Since for strong symplectic Banach manifolds the assumptions (i) and (ii) are
trivially satisfied, we obtain the classical Darboux Theorem for strong
symplectic Banach manifold as a corollary (\textit{cf.} \cite{Ma1} or
\cite{Wei}).


\section{Symplectic form on direct limit of ascending sequence of Banach
manifolds}

\subsection{ Symplectic form on direct limit of an ascending sequence of reflexive Banach spaces}
\label{S_linearSymplectic}${}$\\
Let $\mathbb{E}=\underrightarrow{\lim}\mathbb{E}_{n}$ be a direct limit of an
ascending sequence of reflexive Banach spaces $\left(  \mathbb{E}_{n}\right)
_{n\in\mathbb{N}^{\ast}}$ (\textit{cf.} Appendix
\ref{AscendingSequencesOfSupplementedBanachSpaces}). Let $\left\Vert
\;\right\Vert _{n}$ be a chosen norm on the Banach space $\mathbb{E}_{n}$ such
that
\[
||\;||_{1}\leq\cdots\leq||\;||_{n}\leq\cdots
\]
We consider a sequence $\left(  \omega_{n}\right)  _{n\in\mathbb{N}^{\ast}}$
of symplectic forms (\textit{i.e.} non-degenerate $2$-forms) on $\mathbb{E}%
_{n}$ such that $\omega_{n}$ is the restriction of $\omega_{n+1}$ to
$\mathbb{E}_{n}$ and let $\omega_{n}^{\flat}:\mathbb{E}_{n}\rightarrow
\mathbb{E}_{n}^{\ast}$ be the associated bounded linear operator. Following
the end of section \ref{linearsymplectic}, we consider the norm $||u||_{\omega
_{n}}=||\omega_{n}^{\flat}(u)||_{n}^{\ast}$ where $||\;||_{n}^{\ast}$ is the
canonical norm on $\mathbb{E}_{n}^{\ast}$ associated to $||\;||_{n}$.  We have seen that  
the inclusion of the Banach space
$(\mathbb{E},||\;||_{n})$ in the normed space  $(\mathbb{E}_{n},||\;||_{\omega_{n}})$ is
continuous and we denote by  $\widehat{\mathbb{E}}_{n}$ the Banach space which is the
completion of $(\mathbb{E}_{n},||\;||_{\omega_{n}})$. Recall that from Remark \ref{indnorm}, 
the Banach space  $\widehat{\mathbb{E}}_{n}$ does not depends of the choice of the norm $||\;||_n$ on $\mathbb{E}_n$. 
Again, according to the end of section
\ref{linearsymplectic}, $\omega_{n}^{\flat}$ can be extended to an isometry
between $\hat{\mathbb{E}}_{n}$ and $\mathbb{E}_{n}^{\ast}$. Moreover,
$\omega_{n}^{\flat}$ is an isomorphism from $\mathbb{E}_{n}$ to $\widehat
{\mathbb{E}}_{n}^{\ast}$.\newline

\begin{lemma}
\label{wideEnascending} The family $\widehat{\mathbb{E}}_{n}^{*}$ is an
ascending sequence of Banach spaces and so $\widehat{\mathbb{E}}%
^{*}=\underrightarrow{\lim}\widehat{\mathbb{E}}_{n}^{*}$ is well defined.
\end{lemma}

\begin{proof}
In fact we only have to show that $\widehat{\mathbb{E}}_{n}^{\ast}$ is
contained in $\widehat{\mathbb{E}}_{n+1}^{\ast}$ and the inclusion is
continuous. Since $\omega_{n}^{\flat}$ is an isomorphism from $\mathbb{E}_{n}$
to $\hat{\mathbb{E}}_{n}^{\ast}$ and the restriction of $\omega_{n+1}$ to
$\mathbb{E}_{n}$ is $\omega_{n}$, it follows that the restriction of
$\omega_{n+1}$ to $\mathbb{E}_{n}$ is an isomorphism from $\mathbb{E}_{n}$
onto a Banach subspace of $\widehat{\mathbb{E}}_{n+1}^{\ast}$ which will be
identified with $\widehat{\mathbb{E}}_{n}^{\ast}$ and so the inclusion of
$\widehat{\mathbb{E}}_{n}^{\ast}$ in $\widehat{\mathbb{E}}_{n+1}^{\ast}$ is
continuous which ends the proof.\newline
\end{proof}

\bigskip

Recall that $\mathbb{E}=\underrightarrow{\lim}\mathbb{E}_{n}$ is a convenient
space (see \cite{CaPe} Proposition 20). Given a bounded skew symmetric
bilinear form $\omega$ on $\mathbb{E}$, we again denote $\omega^{\flat}$ the
bounded linear operator from $\mathbb{E}$ to its dual $\mathbb{E}^{\ast}$
defined as usually by $\omega^{\flat}(u)=\omega(u,\;)$. Let $\iota_{n}$ be the
natural inclusion of $\mathbb{E}_{n}$ in $\mathbb{E}$ and set $\omega
_{n}=\iota_{n}^{\ast}\omega$. If $u$ and $v$ belong to $\mathbb{E}$ then $u$
and $v$ belong to some $\mathbb{E}_{n}$ and so we have $\omega(u,v)=\omega
_{n}(u,v)$.

\textit{A symplectic form on }$\mathbb{E}$ is a bounded skew symmetric
bilinear form $\omega$ on $\mathbb{E}$ which is non degenerate where
$\omega^{\flat}$ is injective.\newline

\begin{proposition}
\label{omegan} With the previous notations, we have

\begin{enumerate}
\item[1.] $\omega_{n}=\iota_{n}^{\ast}\omega$ is a weak symplectic form on
$\mathbb{E}_{n}$ and $\omega_{n}^{\flat}=\iota_{n}^{\ast}\circ\omega^{\flat
}\circ\iota_{n}$. Moreover, the sequence $\left(  \omega_{n}\right)
_{n\in\mathbb{N}^{\ast}}$ is coherent, and if $u=\underrightarrow{\lim}u_{n}$
and $v=\underrightarrow{\lim}v_{n}$, then we have $\omega
(u,v)=\underrightarrow{\lim}\omega_{n}(u_{n},v_{n})$. Moreover, we have
$\omega(u,v)=\omega^{\flat}(u)(v)$ where $\omega^{\flat}=\underrightarrow
{\lim}(\omega_{n}^{\flat})$.

\item[2.] Conversely, let $\left(  \omega_{n}\right)  _{n\in\mathbb{N}^{\ast}%
}$ be a coherent sequence of symplectic forms. If $u=\underrightarrow{\lim
}u_{n}$ and $v=\underrightarrow{\lim}v_{n}$, then $\omega
(u,v)=\underrightarrow{\lim}\omega_{n}(u_{n},v_{n})$ is well defined and is a
symplectic form on $\mathbb{E}$.
\end{enumerate}
\end{proposition}

\begin{proof}
${}$\newline\noindent1. Recall that a set $B$ in $\mathbb{E}$  is bounded
if and only if $B$ is a bounded subset of some $\mathbb{E}_{n}$ (\cite{CaPe}
Lemma 19). Consider the subset $B_{n}=\{u_{n}\in\mathbb{E}_{n}\;:\;||u_{n}%
||\leq1\}$, then $\hat{B}=\iota_{n}({B}_{n})$ is a bounded subset of
$\mathbb{E}$ and so there exists $K_{n}>0$ such that $|\omega(u,v)|\leq K_{n}$
for all $u,v\in\hat{B}$; This implies that $|\omega_{n}(u,v)|\leq K_{n}$ for
all $u,v\in B_{n}$ and so $\omega_{n}$ is a continuous skew symmetric bilinear
form on $\mathbb{E}_{n}$. Now we have $\omega_{n}^{\flat}=\iota_{n}^{\ast
}\circ\omega^{\flat}\circ\iota_{n}$. Indeed fix some $v\in\mathbb{E}_{n}$. For
any $u\in\mathbb{E}_{n}$ we have ${}\;\;\;\;\;\;\;\;\;\omega_{n}^{\flat
}(u)(v)=\omega(\iota(u),\iota(v))=\iota_{n}^{\ast}\circ\omega^{\flat}%
(\iota_{n}(u))(v).$ \newline This implies that $\omega_{n}^{\flat}=\iota
_{n}^{\ast}\circ\omega^{\flat}\circ\iota_{n}$ and so $\omega_{n}$ is non
degenerate.\newline Clearly, by construction, the sequence $(\omega_{n})$ is
coherent. Fix some $u,v$ in $\mathbb{E}$. Let $n_{0}$ be the smallest integer
such that $u$ and $v$ belongs to $\mathbb{E}_{n}$. Then for $n\geq n_{0}$, we
have $u_{n}=u_{n_{0}}$ and $v_{n}=v_{n_{0}}$ so $\omega(u,v)=\omega_{n}%
(u_{n},v_{n})$ for all $n\geq n_{0}$. It follows  that $\omega^{\flat
}(u)(v)=\omega_{n}^{\flat}(u_{n})(v_{n})$ for all $n\geq n_{0}$, which implies
the result. Finally, we have already seen that $\omega^{\flat}
=\underrightarrow{\lim}(\omega_{n}^{\flat})$ is well defined. Now, if as
previously, $n_{0}$ is the smallest integer such that $u$ and $v$ belongs to
$\mathbb{E}_{n}$ then for all $n\geq n_{0}$, we have $\omega_{n}^{\flat
}(u)(v)=\omega_{n}(u_{n},v_{n}).$ Since $\omega^{\flat}(u)(v)=\underrightarrow
{\lim}(\omega_{n}^{\flat}(u_{n})(v_{n}),$ this implies the last property.
\newline\noindent2. The same arguments as in the last part of the previous
proof permits to show that $\omega$ is well defined and, by construction, it
is a skew symmetric bilinear form. We have to show that $\omega$ is bounded.
If $B$ is a bounded set of $\mathbb{E}$ then $B$ is contained in some
$\mathbb{E}_{n}$ and so $\omega(u,v)=\omega_{n}(u,v)$ when $u$ and $v$ belongs
to $\mathbb{E}_{n}$. Since $\omega_{n}$ is continuous, there exists $K>0$ such
that $|\omega_{n}(u,v)|\leq K$  for  $u,v\in B$ and so $\omega$ is a bounded bilinear form. It
remains to show that $\omega^{\flat}$ is injective. Assume that $\omega
(u,v)=0$ for all $v\in\mathbb{E}$. Let $n_{0}$ the smallest integer such that
$u$ belongs to $E_{n_{0}}$. Our assumption implies that $\omega(u,v)=0$ for
all $v\in\mathbb{E}_{n}$ and all $n\geq n_{0}$. But $\omega(u,v)=\omega
_{n}(u,v)$ for all $v\in\mathbb{E}_{n}$. Since each $\omega_{n}$ is
symplectic, this implies that $u=0$.\newline${}\hfill$
\end{proof}


\subsection{Symplectic form on a direct limit of an ascending sequence of
Banach bundles}

\label{symplectic-banach-bundles}${}$\\
We consider an ascending sequence $\left\{  \left(  E_{n},\pi_{n}%
,M_{n}\right)  \right\}  _{n\in\mathbb{N}^{\ast}}$ of Banach vector bundles
(\textit{cf.} Appendix \ref{DirectLimitsOfBanachVectorBundles}) where the
typical fiber $\mathbb{E}_{n}$ is reflexive. The direct limit 
$E=\underrightarrow{\lim}E_{n}$ has a structure of n.n.H. convenient bundle
over $M=\underrightarrow{\lim}M_{n}$ with typical fiber $\mathbb{E}%
=\underrightarrow{\lim}\mathbb{E}_{n}$ (\textit{cf.} \cite{CaPe} Proposition
\ref{P_StructureOnDirectLimitLinearBundles}).\newline

A sequence $\left(  \omega_{n}\right)  _{n\in\mathbb{N}^{\ast}}$ of symplectic
forms $\omega_{n}$ on $E_{n}$ for $n\in\mathbb{N}^{\ast}$ is called \textit{an
ascending sequence of symplectic forms} or is \textit{coherent for short} if
we have
\[
\forall n\in\mathbb{N}^{\ast},\ (\lambda_{n}^{n+1})^{\ast}\omega_{n+1}%
\circ\epsilon_{n}^{n+1}=\omega_{n}.\newline%
\]
where $\lambda_{n}^{n+1}$ (resp. $\epsilon_{n}^{n+1}$) is the natural
inclusion of $E_{n}$ (resp. $M_{n}$) in $E_{n+1}$ (resp. $M_{n+1}$) (see the
notations introduced in Appendix \ref{direct_limit_Banach_manifolds_bundles}%
).\newline

Consider a coherent sequence $\left(  \omega_{n}\right)  _{n\in\mathbb{N}%
^{\ast}}$ of symplectic forms on $E_{n}$. According to the notations of
section \ref{DarbouxBanach}, since $\mathbb{E}_{n}$ is reflexive, we denote by
$(\widehat{E_{n}})_{x_{n}}$ the Banach space which is the completion of
$(E_{n})_{x_{n}}$ provided with the norm $||\;||_{(\omega_{_{n}})_{x_{n}}}$.
Then $(\omega_{n})_{x_{n}}$ can be extended to a continuous bilinear map
$(\hat{\omega}_{n})_{x_{n}}$ on $(E_{n})_{x_{n}}\times(\widehat{E_{n}}%
)_{x_{n}}$ and $(\omega_{n})_{x_{n}}^{\flat}$ becomes an isomorphism from
$(E_{n})_{x_{n}}$ to $(\widehat{E_{n}})_{x_{n}}^{\ast}$. We set
\[
\widehat{E_{n}^{\ast}}=\bigcup_{x_{n}\in M_{n}}(\widehat{E_{n}^{\ast}}%
)_{x_{n}}%
\]

\begin{proposition}
${}$ \label{omega_ncoherent}

\begin{enumerate}
\item[1.] In the previous context, consider a coherent sequence $\left(
\omega_{n}\right)  _{n\in\mathbb{N}^{\ast}}$ of symplectic forms $\omega_{n}$
on $E_{n}$. Then, for each $x=\underrightarrow{\lim}x_{n}\in M$, if
$u=\underrightarrow{\lim}u_{n}$ and $v=\underrightarrow{\lim}v_{n}$ belong to
$E_{x}$, then we have $\omega_{x}(u,v)=\underrightarrow{\lim}(\omega
_{n})_{x_{n}}(u_{n},v_{n})$ and we get a smooth symplectic form on $E$.
Moreover, we have $\omega_{x}(u,v)=\omega_{x}^{\flat}(u)(v)$ where $\omega
_{x}^{\flat}=\underrightarrow{\lim}(\omega_{n}^{\flat})_{x_{n}}$.

\item[2.] Let $\omega$ be a symplectic form on $M$. The inclusion
$\varepsilon_{n}^{n+1}:M_{n}\rightarrow M_{n+1}$ (resp. the map $\lambda
_{n}^{n+1}:E_{n}\rightarrow E_{n+1}$) defines a smooth injective map
$\varepsilon_{n}:M_{n}\rightarrow M$ (resp. an injective convenient bundle
morphism $\lambda_{n}:E_{n}\rightarrow E$). If we set $\omega_{n}=\lambda
_{n}^{\ast}(\omega)\circ\varepsilon_{n}$, then $\omega_{n}$ is a symplectic
form on $E_{n}$ and the sequence $\left(  \omega_{n}\right)  _{n\in
\mathbb{N}^{\ast}}$ is a family of coherent symplectic forms. Moreover the
symplectic form on $M$ associated to this family is exactly $\omega$.\\
\end{enumerate}
\end{proposition}

\begin{proof}${}$\\
\noindent1. Consider $x=\underrightarrow{\lim}x_{n}\in M$ and
$u=\underrightarrow{\lim}u_{n}$ and $v=\underrightarrow{\lim}v_{n}$ in $E_{x}%
$. Let $n_{0}$ be the smallest integer such that $x$ belongs to $M_{n}$ and
$u,v$ belongs to $E_{x_{n}}$. Then, for all $n\geq n_{0}$, we have
$x_{n}=x_{n_{0}}$, $\;u_{n}=u_{n_{0}}$ and $v_{n}=v_{n_{0}}$. This implies
that $\omega_{x}(u,v)=(\omega_{n})_{x_{n}}(u_{n},v_{n})=(\omega_{n_{0}%
})_{x_{n_{0}}}(u_{n_{0}},v_{n_{0}})$. From Proposition \ref{omegan}, it
follows that $\omega_{x}$ is a symplectic form on $E_{x}$ and $\omega
_{x}(u,v)=\omega_{x}^{\flat}(u)(v)$. Consider any bundle chart $\left(
U=\underrightarrow{\lim}U_{n},\Phi=\underrightarrow{\lim}\Phi_{n}\right)  $,
where $\Phi_{n}$ is a trivialization of $(E_{n})_{|U_{n}}$ (\textit{cf.}
Appendix \ref{DirectLimitsOfBanachVectorBundles}). Since $(x_{n},u_{n}%
,v_{n})\mapsto(\omega_{n})_{x_{n}}(u_{n},v_{n})$ is a smooth map from
$(E_{n})_{|U_{n}}$ to $\mathbb{R}$, it follows that $(x,u,v)\mapsto\omega
_{x}(u,v)$ is a smooth map from $E_{|U}$ to $\mathbb{R}$ (\textit{cf.}
\cite{CaPe} Lemma 22). This implies that $\omega$ is a smooth symplectic
form.\newline2. Let $\omega$ be a symplectic form on $E$. From the universal
properties of a direct limit and because $\varepsilon_{n}^{n+1}:M_{n}%
\rightarrow M_{n+1}$ is an injective smooth map (resp. $\lambda_{n}%
^{n+1}:E_{n}\rightarrow E_{n+1}$ is an injective Banach bundle morphism), we
obtain a smooth injective map $\varepsilon_{n}:M_{n}\rightarrow M$ (resp. an
injective convenient bundle morphism $\lambda_{n}:E_{n}\rightarrow E$)
(\textit{cf.} \cite{CaPe} Lemma 22 again). This implies that $\omega_{n}$ is a
non degenerate $2$-form on $E_{n}$.
But from the definition of $\omega_{n}$, it is clear that $(\omega_{n})$ is a
sequence of coherent symplectic forms. Now, by application of the proof of
Point 1 to the sequence $(\omega_{n})$, the symplectic form defined by the
sequence $(\omega_{n})$ is clearly the given symplectic form $\omega$%
.\newline${}\hfill$
\end{proof}

\noindent According to the assumption of Theorem \ref{Moserlemma} we introduce the
following terminology:\newline

\begin{definition}
\label{Darbouxhyp} Let $\left\{  \left(  E_{n},\pi_{n},M_{n}\right)  \right\}
_{n\in\mathbb{N}^{\ast}}$ be a sequence of strong ascending Banach bundles
whose typical fiber $\mathbb{E}_{n}$ is reflexive. Consider a coherent
sequence $\left(  \omega_{n}\right)  _{n\in\mathbb{N}^{\ast}}$ of symplectic
forms $\omega_{n}$ on $E_{n}$. We say that the sequence $\left(  \omega
_{n}\right)  _{n\in\mathbb{N}^{\ast}}$ satisfies the Bambusi-Darboux assumption around
$x^{0}\in M$ if there exists a domain of a direct limit chart
$U=\underrightarrow{\lim}U_{n}$ around $x^{0}$ such that:

\begin{description}
\item[(i)] For each $n>0$, $(\widehat{E}_{n})_{|U_{n}}$ is a trivial Banach
bundle;

\item[(ii) ] For each $n>0$, $\;\omega_{n}$ can be extended to a smooth field
of continuous bilinear forms on $(E_{n})_{|U_{n}}\times(\widehat{E}%
_{n})_{|U_{n}}$.\newline
\end{description}
\end{definition}

Under these assumptions we have:

\begin{proposition}
\label{firstproperties} Consider a coherent sequence $\left(  \omega
_{n}\right)  _{n\in\mathbb{N}^{\ast}}$ of symplectic forms $\omega_{n}$ on
$E_{n}$ which satisfies the Bambusi-Darboux assumption around $x^{0}\in M$. Then we
have the following properties

\begin{enumerate}
\item[1.] The direct limit $\widehat{E}^{*}_{|U}=\underrightarrow{\lim
}\widehat{E_{n}}^{*}_{| U_{n}}$ is well defined and is a trivial convenient
bundle with typical fiber $\widehat{\mathbb{E}}^{*}=\underrightarrow{\lim
}\widehat{\mathbb{E}}^{*}_{n}$.

\item[2] The sequence $\left(  \omega_{n}^{\flat}\right)  _{n\in
\mathbb{N}^{\ast}}$of isomorphisms from $E_{n}$ to $\widehat{E}_{n}^{\ast}$
induces an isomorphism from $E_{|U}$ to $\widehat{E}_{|U}^{\ast}$.\newline
\end{enumerate}
\end{proposition}

\begin{proof}
From our assumptions, for each $n>0$, we have a sequence of trivializations
$\widehat{\Phi}_{n}:(\widehat{E}_{n})_{|U_{n}}\rightarrow U_{n}\times
\widehat{\mathbb{E}}_{n}$. Thus we obtain a sequence $\widehat{\Phi}_{n}%
^{-1}:U_{n}\times\widehat{\mathbb{E}}_{n}^{\ast}\rightarrow(\widehat{E}%
_{n})_{|U_{n}}^{\ast}$ of isomorphisms of trivial bundles. Now, we have
natural inclusions $\widehat{\iota}_{n}^{n+1}:$ $\widehat{\mathbb{E}}%
_{n}^{\ast}$ $\longrightarrow$ $\widehat{\mathbb{E}}_{n+1}^{\ast}$ and
$\varepsilon_{n}^{n+1}$ : $U_{n}$ $U_{n+1}$. So we get an injective bundle
morphism $\delta_{n}^{n+1}$ from $U_{n}\times\widehat{\mathbb{E}}_{n}^{\ast}$
to $U_{n+1}\times\widehat{\mathbb{E}}_{n+1}^{\ast}$. Therefore the map defined
by $\widehat{\lambda}_{n}^{n+1}(x,u)=(\Phi_{n+1}^{\ast})_{x}\circ\delta
_{n}^{n+1}\circ((\Phi_{n})_{n}^{\ast})^{-1})_{x}(u)$ for all $(x,u)\in
(\widehat{E}_{n})_{|U_{n}}$ is a bonding map for the ascending sequence of
trivial bundles $\left(  (\widehat{E}_{n}^{\ast})_{|U_{n}}\right)
_{n\in\mathbb{N}^{\ast}}$. Therefore the direct limits $\widehat{\Phi}^{\ast
}=\underrightarrow{\lim}\widehat{\Phi}_{n}^{\ast}$ and $\widehat{E}%
_{|U}=\underrightarrow{\lim}(\widehat{E}_{n}^{\ast})_{|U_{n}}$ are well
defined and $\widehat{\Phi}$ is a convenient isomorphism bundle from
$U\times\widehat{\mathbb{E}}$ to $\widehat{E}_{|U}$ which ends the proof of
Point 1.\newline\noindent2. At first, from Proposition \ref{omega_ncoherent}
Point 1, the sequence $\left(  \omega_{n}\right)  _{n\in\mathbb{N}^{\ast}}$
defines a symplectic form on $E$. From our assumption, since for each $n>0$ we
can extend $\omega_{n}$ to a bilinear onto $(E_{n})_{|U_{n}}\times(\widehat
{E}_{n})_{|U_{n}}$, this implies that $\omega_{n}^{\flat}$ is an isomorphism
from $(E_{n})_{|U_{n}}$ to $(\widehat{E}_{n}^{\ast})_{|U_{n}}$. Consider the
sequence of bonding maps $\widehat{\lambda}_{n}^{n+1}$ for the ascending
sequence of $\left(  (\widehat{E}_{n}^{\ast})_{|U_{n}}\right)  _{n\in
\mathbb{N}^{\ast}}$ previously defined. Then we have the following commutative
diagram:
\[%
\begin{array}
[c]{ccccccc}%
U\times\widehat{E}_{n+1}^{\ast} & \underleftarrow{\Phi_{n+1}^{\ast}} &
(\widehat{E}_{n+1}^{\ast})_{|U} & \underleftarrow{\omega_{n+1}^{\flat}} &
(E_{n+1})_{|U} & \underrightarrow{\Phi_{n+1}} & U\times\mathbb{E}_{n+1}\\
\uparrow &  & \uparrow\widehat{\lambda}_{n}^{n+1} &  & \uparrow\lambda
_{n}^{n+1} &  & \uparrow\\
U\times\widehat{E}_{n}^{\ast} & \underleftarrow{\Phi_{n}^{\ast}} &
(\widehat{E}_{n}^{\ast})_{|U} & \underleftarrow{\omega_{n}^{\flat}} &
(E_{n})_{|U} & \underrightarrow{\Phi_{n}} & U\times\mathbb{E}_{n}\\
&  &  &  &  &  &
\end{array}
\]
\noindent It follows that the direct limit $\omega^{\flat}=\underrightarrow
{\lim}\omega_{n}^{\flat}$ is well defined and is an isomorphism from $E_{|U}$
to $\widehat{E}_{|U}^{\ast}$ ${}\hfill$
\end{proof}


\subsection{Problem of existence of Darboux charts on a strong ascending
sequence of Banach manifolds}


\subsubsection{Conditions of existence of Darboux charts}

Let $\left\{  (M_{n},\varepsilon_{n}^{n+1})\right\}  _{n\in\mathbb{N}^{\ast}}$
be an ascending sequence of Banach manifolds where $M_{n}$ is modelled on a
Banach space $\mathbb{M}_{n}$ (\textit{cf.} Appendix
\ref{DirectLimitsOfAscendingSequencesOfBanachManifolds}). \newline

According to Corollary \ref{C_DirectLimitTangentBundle}, the sequence $\left(
TM_{n},\pi_{n},M_{n}\right)  _{n\in\mathbb{N}^{\ast}}$ associated to an
ascending sequence $\left\{  (M_{n},\varepsilon_{n}^{n+1})\right\}
_{n\in\mathbb{N}^{\ast}}$ has a direct limit $TM=\underrightarrow{\lim}TM_{n}$
which has a structure of convenient vector bundle whose typical fiber is
$\underrightarrow{\lim}\mathbb{\mathbb{M}}_{n}$ over $M=\underrightarrow{\lim
}M_{n}$.

\noindent Now, since $M$ is a convenient manifold, a \textit{symplectic form
on} $M$ is a differential $2$-form on $M$ which is non degenerate and such
that $d\omega=0$ (\textit{cf.} \cite{KrMi} 48).

\begin{theorem}
\label{omega_ncoherentM}${}$
\begin{enumerate}
\item[1.] In the previous context, consider a coherent sequence $\left(
\omega_{n}\right)  _{n\in\mathbb{N}^{\ast}}$ of symplectic forms $\omega_{n}$
on $M_{n}$\footnote{That is $(\omega_{n})$ is a coherent sequence of
symplectic forms on the bundles $TM_{n}$ as defined in section
\ref{symplectic-banach-bundles}}. Then, for each $x\in M$, the direct limit
$\omega_{x}^{\flat}=\underrightarrow{\lim}(\omega_{n})_{x_{n}}^{\flat}$ is
well defined and is an isomorphism from $T_{x}M$ to $(\widehat{T_{x}}M)^{\ast
}$. Moreover $\omega_{x}(u,v)=\omega_{x}^{\flat}(u)(v)$ defines a smooth
symplectic form on $M$.

\item[2.] Let $\omega$ be a symplectic form on $M$. The inclusion
$\varepsilon_{n}:M_{n}\rightarrow M$ is a smooth injective map. If we set
$\omega_{n}=\varepsilon_{n}^{\ast}(\omega)$, then $\omega_{n}$ is a symplectic
form on $M_{n}$ and the sequence $\left(  \omega_{n}\right)  _{n\in
\mathbb{N}^{\ast}}$ is a sequence of coherent symplectic forms. Moreover the
symplectic form on $M$ associated to this sequence is exactly $\omega$.
\end{enumerate}
\end{theorem}

\begin{proof}${}$\\
\noindent1. By application of Point 1 of Proposition \ref{omega_ncoherent} to
$E_{n}=TM_{n}$, we obtain the first part. It only remains to prove that
$d\omega=0$. Let $\varepsilon_{n}$ be the inclusion of $M_{n}$ in $M$. Since
each inclusion $\varepsilon_{n}^{j}$ : $M_{n}$ $\longrightarrow$ $M_{j}$ is a
smooth injective map, so is $\varepsilon_{n}$. Therefore we have
\[
\varepsilon_{n}^{\ast}(d\omega)=d_{n}\varepsilon_{n}^{\ast}(\omega
)=d_{n}\omega_{n}%
\]
where $d_{n}$ is the exterior differential on $M_{n}$ (cf. \cite{KrMi}, 33),
consider $u=\underrightarrow{\lim}u_{n}$ and $v=\underrightarrow{\lim}v_{n}$
in $T_{x}M$ where $x=\underrightarrow{\lim}x_{n}$. Since as a set, we have
$T_{x}M=\bigcup\limits_{n\in\mathbb{N}^{\ast}}T_{x_{n}}M_{n}$, let $n_{0}$ be
the smallest integer $n$ such that $u$ and $v$ belong to $T_{x_{n}}M_{n}$.
Then, for all $n\geq n_{0}$, we have%

\begin{align*}
d\omega_{x}(u,v)  & =d\omega_{x}(T\varepsilon_{n}(u),T\varepsilon
_{n}(v))=\varepsilon_{n}^{\ast}(d\omega_{x})(u,v)\\
& =d_{n_{0}}\varepsilon_{n}^{\ast}(\omega_{x})(u,v){}=d_{n}(\omega_{n}%
)_{x_{n}}(u_{n},v_{n})\\
& =0.
\end{align*}
\noindent As $u=\underrightarrow{\lim}u_{n}$ and $v=\underrightarrow{\lim
}v_{n}$, we have $(\omega_{n})_{x}(u_{n},v_{n})=(\omega_{n_{0}})_{x_{n_{0}}%
}(u_{n_{0}},v_{n_{0}})$ for all $n\geq n_{0}$, which ends the proof of Point
1.\newline\noindent2. Again the first part is a consequence of Proposition
\ref{omega_ncoherent}. Now since $d\omega=0$, we obtain $d_{n}\omega_{n}=0$.
\newline But from the definition of $\omega_{n}$, it is clear that $\left(
\omega_{n}\right)  $ is a sequence of coherent symplectic forms. Now, by
application of the proof of Point 1 to the sequence $\left(  \omega
_{n}\right)  $, the symplectic form defined by $\left(  \omega_{n}\right)  $
is clearly the given symplectic form $\omega$.
\end{proof}

As in the Banach context, we introduce the notion of Darboux chart:

\begin{definition}
\label{ascendingdarbouxchart} Let $\omega$ be a weak symplectic form on the
direct limit $M=\underrightarrow{\lim}M_{n}$. We say thet a chart $(V,\psi)$
around $x_{0}$ is a Darboux chart if $\psi^{\ast}\omega^{0}=\omega$ where
$\omega^{0}$ is the constant form on $\psi(U)$ defined by $(\psi^{-1})^{\ast
}\omega_{x_{0}}$.
\end{definition}

We have the following necessary and sufficient conditions of existence of
Darboux charts on a direct limit of ascending Banach manifolds:

\begin{theorem}
\label{equivDarbouxascending}Let $\left\{  (M_{n},\varepsilon_{n}%
^{n+1})\right\}  _{n\in\mathbb{N}^{\ast}}$ be an ascending sequence of Banach
manifolds where $M_{n}$ is modelled on a reflexive Banach space $\mathbb{M}%
_{n}$.

\begin{enumerate}
\item[1.] Consider a coherent sequence $\left(  \omega_{n}\right)
_{n\in\mathbb{N}^{\ast}}$ of symplectic forms $\omega_{n}$ on $M_{n}$ and let
$\omega$ be the symplectic form generated by $\left(  \omega_{n}\right)  $ on
$M=\underrightarrow{\lim}M_{n}$. Assume the following property is satisfied:

(H) there exists a limit chart $(U=\underrightarrow{\lim}U_{n},\phi
=\underrightarrow{\lim}\phi_{n})$ around $x_{0}$ such that if $x_{0}$ belongs
to $M_{n}$, then $(U_{n},\phi_{n})$ is a Darboux chart around $x_{0}$ for
$\omega_{n}$.

Then $(U,\phi)$ is a Darboux chart around $x_{0}$ for $\omega$.

\item[2.] Let $\omega$ be a weak symplectic form on the direct limit
$M=\underrightarrow{\lim}M_{n}$. Assume that there exists a Darboux chart
$(V,\psi)$ around $x_{0}$ in $M$.\ If $\omega_{n}$ is the symplectic form on
$M_{n}$ induced by $\omega$, then there exists a limit chart
$(U=\underrightarrow{\lim}U_{n},\phi=\underrightarrow{\lim}\phi_{n})$ around
$x_{0}$ such that the property (H) is satisfied.\newline
\end{enumerate}
\end{theorem}

\begin{proof}
1. Assume that the assumption (H) is true. Note that by construction of
$\omega$, we have $\omega_{n}=\epsilon_{n}^{\ast}\omega$. Fix some $x\in U$
and and $u,v$ in $T_{x}M$. Denote by $n_{0}$ the small integer such $x$
belongs to $M_{n}$ and $u,v$ are in $T_{x}M_{n}$. Then for any $n\geq n_{0}$,
we have $\omega_{x}(u,v)=(\omega_{n})_{x}(u,v)$. For the same reason, if
$\omega_{n}^{0}$ is the "model" in the Darboux chart for $\omega_{n}$, we also
have $\omega^{0}(u,v)=\omega_{n}^{0}(u,v)$. From the properties of
compatibility of $\phi$ and $\phi_{n}$, the induced form on $\phi_{n}(U_{n})$
by $\phi^{\ast}\omega$ is precisely $\phi_{n}^{\ast}(\omega_{n})$ and moreover
the constant $2$-form on $\phi_{n}(U_{n})$ induced by $\omega^{0}$ is
$\omega_{n}^{0}$. But according to our assumption, we have $\omega_{n}%
^{0}(T_{x}\phi_{n}(u),T_{x}\phi_{n}(v)=\omega_{n}^{0}((T_{x}\phi_{n}%
(u),T_{x}\phi_{n}(v))$. \noindent Thanks to the compatibility conditions
between $T_{x}\phi$ and $T_{x}\phi_{n}$, we then obtain:
\begin{align*}
\omega(T_{x}\phi(u),T_{x}\phi(v)  & =\omega_{n}^{0}(T_{x}\phi_{n}(u),T_{x}%
\phi_{n}(v)\\
& =\omega_{n}^{0}((T_{x}\phi_{n}(u),T_{x}\phi_{n}(v))\newline{}\\
& =\omega^{0}(T_{x}\phi(u),T_{x}\phi(v)\text{.}%
\end{align*}
\noindent Since $T_{x}\phi_{n}$ is an isomorphism onto $\mathbb{M}_{n}$, this
ends the proof of Point 1.\\

\noindent2. Given a symplectic form $\omega
$, we can consider the induced form $\omega_{n}=\varepsilon_{n}^{\ast}\omega$
on $M_{n}$. Consider a Darboux chart $(V,\psi)$ for $\omega$ around $x_{0}$.
We set $V_{n}=V\cap M_{n}$ and $\psi_{n}=\psi\circ\epsilon_{n}$. From the
condition of compatibility of $\psi$ and the inclusion of $\mathbb{M}_{n}$ in
$\mathbb{M}$, it follows that $\psi(V\cap M_{n})=\psi_{n}(V_{n})$ and
$\psi_{n}$ is a diffeomorphism from $V_{n}$ to $\psi(V_{n})$. Therefore we
have $V=\underrightarrow{\lim}V_{n}$ and $\psi=\underrightarrow{\lim}\psi_{n}%
$. Using the same argument as in the first part, we show that $(V_{n},\psi
_{n})$ is a Darboux chart for $\omega_{n}$.

\end{proof}

\bigskip
In section \ref{symplecasceloop}, the Example \ref{exDarbouxascending} is a
situation of an \textbf{ascending sequence} of Sobolev manifolds of loops on
which \textbf{a Darboux chart exists}. The same result is true if there there
exists an integer $n_{1}$ such that $M_{n}=M_{n_{1}}$ for all $n\geq n_{1}$.
However, in general, without very particular situations, such a sequence of
Darboux chart $\left\{  (V_{n},\psi_{n})\right\}  _{n\in\mathbb{N}^{\ast}}$
will satisfy $\bigcap\limits_{n\geqslant n_{0}}W_{n}=\{x_{0}\}$, \textbf{even}
if each $\omega_{n}$ is a strong symplectic form, according to the Example
\ref{nochartDarboux}. For such a type of discussion see the next section.


\subsubsection{Problem of existence of Darboux chart in general}

\label{no_Darboux_ascending}${}$\\

{ In this subsection, we will explain why, even in the context of an
ascending sequence of symplectic Banach manifolds which satisfies the
assumption of Theorem \ref{localDarboux}, in general, there does not exist
Darboux charts for the induced symplectic form on the direct limit.}\newline

Let $\left\{  (M_{n},\varepsilon_{n}^{n+1})\right\}  _{n\in\mathbb{N}^{\ast}}$
be an ascending sequence of Banach manifolds where $M_{n}$ is modelled on a
\textit{reflexive} Banach space $\mathbb{M}_{n}$.
Consider a coherent sequence $\left(  \omega_{n}\right)  _{n\in\mathbb{N}%
^{\ast}}$ of symplectic forms on $M_{n}$. According to the notations of
section \ref{S_linearSymplectic}, since $\mathbb{M}_{n}$ is reflexive, we
denote by $\widehat{T_{x_{n}}M_{n}}$ the Banach space which is the completion
of $T_{x_{n}}M_{n}$ provided with the norm $||\;||_{(\omega_{n})_{x_{n}}}$.
Then $(\omega_{n})_{x_{n}}$ can be extended to a continuous bilinear map
$(\hat{\omega}_{n})_{x_{n}}$ on $T_{x_{n}}M_{n}\times\widehat{T_{x_{n}}M_{n}}$
and $(\omega_{n})_{x_{n}}^{\flat}$ becomes an isomorphism from $T_{x_{n}}%
M_{n}$ to $(\widehat{T_{x_{n}}M_{n}})^{\ast}$. We set
\[
(\widehat{TM_{n}})^{\ast}=\bigcup\limits_{x_{n}\in M_{n}}(\widehat{T_{x_{n}%
}M_{n}})^{\ast}\text{.}%
\]
.

Then by application of Proposition \ref{firstproperties}, we have:

\begin{proposition}
\label{Darbouxascending}Let $\left\{  (M_{n},\varepsilon_{n}^{n+1})\right\}
_{n\in\mathbb{N}^{\ast}}$ be an ascending sequence of Banach manifolds whose
model is a reflexive Banach space $\mathbb{M}_{n}$. Consider a coherent
sequence $\left(  \omega_{n}\right)  _{n\in\mathbb{N}^{\ast}}$ of symplectic
forms $\omega_{n}$ on $E_{n}$. Assume that we have the following
assumption\footnote{This assumption corresponds to the Bambusi-Darboux assumption in
Definition \ref{Darbouxhyp}}:

\begin{enumerate}
\item[(i)] There exists a limit chart $(U=\underrightarrow{\lim}U_{n}%
,\phi=\underrightarrow{\lim}\phi_{n})$ around $x_{0}$ such that $(\widehat
{TM}_{n})_{|U_{n}}$ is a trivial Banach bundle.

\item[(ii)] $\;\omega_{n}$ can be extended to a smooth field of continuous
bilinear forms on $(TM_{n})_{|U_{n}}\times(\widehat{TM}_{n})_{|U_{n}}$.
\end{enumerate}

Then $\widehat{T^{\ast}M}_{|U}$ is a trivial bundle. If $\omega$ is the
symplectic form defined by the sequence $\left(  \omega_{n}\right)
_{n\in\mathbb{N}^{\ast}}$, then the morphism
\[
\omega^{\flat}:TM\rightarrow T^{\ast}M
\]
induces an isomorphism from $TM_{|U}$ to $\widehat{T^{\ast}M}_{|U}$.
\end{proposition}

Note that the context of Proposition \ref{Darbouxascending} covers the
particular framework of sequences of strong symplectic Banach manifolds
$(M_{n},\omega_{n})$.\newline

\textit{We will expose which arguments are needed to prove  a Darboux theorem in the context of direct limit
of Banach manifolds under the assumption of Proposition \ref{Darbouxascending}%
. In fact, we point out the problems that arise in establishing the existence
of a Darboux chart by Moser's method. According to Theorem \ref{existsol}, the
precise result that we can obtain in this way needs so strong assumptions on
}$\omega$\textit{ that we are not sure that this result can have non trivial
applications; so we omit to give such a precise Darboux Theorem (see Remark
\ref{positiveresult})}.\newline

Fix some point $a=\underrightarrow{\lim}a_{n}\in M$. In the context on
Proposition \ref{Darbouxascending}, as in the proof of Theorem
\ref{localDarboux}, on the direct limit chart $(U,\phi)$ around $a$, we can
replace $U$ by $\phi(U)$, $\omega$ by $(\phi)^{\ast}\omega$ on $\phi(U)$ in
the convenient space $\mathbb{M}$. Thus, if $\omega^{0}$ is the constant form
on $U$ defined by $\omega_{a}$, we consider the $1$-parameter family
\[
\omega^{t}=\omega^{0}+t\overline{\omega},\text{ with }\overline{\omega}%
=\omega-\omega^{0}.
\]
Since $\omega^{t}$ is closed and $\mathbb{M}$ is a convenient space, by
\cite{KrMi} Lemma 33.20, there exists a neighbourhood $V\subset U$ of $a$ and
a $1$-form $\alpha$ on $V$ such that $\alpha=d\overline{\omega}$ which is
given by
\[
\alpha_{x}:=\int_{0}^{1}s.\overline{\omega}_{sx}(x,\;)ds.
\]
Now, for all $0\leq t\leq1$, $\omega_{x_{0}}^{t}$ is an isomorphism from
$T_{a}M\equiv\mathbb{M}$ onto $\widehat{T_{a}M}\equiv\widehat{\mathbb{M}%
}^{\ast}$. In the Banach context, using the fact that the set of invertible
operators is open in the set of operators, after restricting $V$, we may
assume that $(\omega^{t})^{\flat}$ is a field of isomorphisms from
$\mathbb{M}$ to $\widehat{\mathbb{M}}^{\ast}$. \newline

Unfortunately, \textit{this result is not true in the convenient setting}
(\textit{cf.} \cite{KrMi}). Therefore, the classical proof does not work in
this way. \newline

However, let $\omega_{n}$ be the symplectic form induced by $\omega$ on
$M_{n}$. Therefore, for each $n$, let $\alpha_{n}$ be the $1$-form $\alpha
_{n}$ induced by $\alpha$ on $\phi_{n}(U_{n}\cap V))$. Then we have
$\omega_{n}=d\alpha_{n}$ and also
\[
(\alpha_{n})_{x_{n}}=\int_{0}^{1}s.(\overline{\omega}_{n})_{sx_{n}}%
(x_{n},\;)ds.
\]
where $\bar{\omega}_{n}=\omega_{n}-\omega_{n}^{0}$ is associated to the
$1$-parameter family $\omega_{n}^{t}=\omega^{n}+t\bar{\omega}_{n}$. We are
exactly in the context of the proof of Theorem \ref{localDarboux} and so the
local flow $F_{n}^{t}$ of $X_{n}^{t}=((\omega_{n}^{t})^{\flat})^{-1}%
(\alpha_{n})$ is a local diffeomorphism from a neighbourhood $W_{n}$ of $a$ in
$V_{n}$ and, in this way, we build a Darboux chart around $a_{n}$ in $M_{n}$.
Therefore, after restricting each $W_{n}$, \textbf{assume} that:\\

{\it("direct limit Darboux chart")  we have an
ascending sequence of such open sets $\{W_{n}\}_{n\in \N}$ then on $W=\underrightarrow
{\lim}W_{n}$, the family of local diffeomorphisms $F^{t}=\underrightarrow
{\lim}F_{n}^{t}$ is defined on $W$.} \\

 Recall that $\omega^{\flat}=\underrightarrow{\lim}\omega_{n}%
^{\flat}$ and $\omega^{\flat}$ is an isomorphism. Thus according to the
previous notations, we have a time dependent vector field
\[
X^{t}=((\omega^{t})^{\flat})^{-1}(\alpha)
\]
and again, we have $L_{X^{t}}\omega^{t}=0$. Of course, if the "direct limit Darboux chart"
assumption on $\{W_{n}\}_{n\in \N}$ is true, then $X^{t}=\underrightarrow{\lim}X_{n}^{t}$. 
So we obtain a Darboux chart as in the Banach context. Note that, in this
case, we are in the context of Theorem \ref{equivDarbouxascending} .\newline

\begin{remark}\label{pbDarbouxchart}\normalfont
In fact, under  the "direct limit Darboux chart"  assumption  the fow $F^{t}$ is the local flow (at
time $t\in\lbrack0,1]$) of $X^{t}=\underrightarrow{\lim}X_{n}^{t}$ where $X_{n}^{t}=((\omega_{n}^{t})^{\flat})^{-1}
(\alpha_{n})$ (with the previous notations).
Unfortunately, the "Darboux chart" assumption is not true since we have
$\bigcap\limits_{n\in\mathbb{N}^{\ast}}W_{n}=\{a\}$  in general (see Appendix
\ref{sufsolODE} and Appendix \ref{comments}).
\end{remark}

\begin{remark}
\label{positiveresult} \normalfont With the previous notations, if $X_{n}^{t}$ satisfies
the assumptions of Theorem \ref{existsol}, then we can define the flow of
$X^{t}=\underrightarrow{\lim}X_{n}^{t}$ on some neighbourhood $W$ of $a$ and,
in this way, we could produce a Theorem of existence of a Darboux chart for
$\omega$ around $a$. Of course, it is clear that we must translate the
conditions ($A_{n}$), (B), (C) of this theorem in terms of conditions on
$\left(  \omega_{n}\right)  $ or on $\omega$ which gives rise to a really non
applicable result even for a direct limit of finite dimensional symplectic
manifolds (\textit{cf.} Examples \ref{exexistsol} 3-4).\newline
\end{remark}

${}\;\;$From the example of \cite{Ma2}, we can obtain the following example
for which \textbf{there is no Darboux chart} on a direct limit of symplectic
Banach manifolds:

\begin{example}\normalfont
\label{nochartDarboux}Let $\mathbb{H}$ be a Hilbert space and endowed with its
inner product $<\;,\;>$. If $g$ is a Riemannian metric on $T\mathbb{H}%
=\mathbb{H}\times\mathbb{H}$, we can define a symplectic form $\omega$ in the
following way (\cite{Ma2}):
\[
2\;\omega_{u,e}((e_{1},e_{2}),(e_{3},e_{4}))=D_{u}g_{u}(e,e_{2}).e_{3}%
-D_{u}g_{u}(e,e_{3}).e_{1}+g_{u}(e_{4},e_{1})-g_{u}(e_{2},e_{3}).
\]
$\hfill$

Let $S:\mathbb{H}\rightarrow\mathbb{H}$ be a compact operator with dense
range, but proper subset of $\mathbb{H}$, which is selfadjoint and positive.
Given a fixed $e\in\mathbb{H}$, then $A_{x}=||x-e||^{2}Id_{\mathbb{H}}+S$ is a
smooth field of bounded operators of $\mathbb{H}$ which is an isomorphism for
all $x\not =e$ and $A_{e}(\mathbb{H})\not =\mathbb{H}$ but  $A_{e}(\mathbb{H})$
 is dense in
$\mathbb{H}$. Then $g_{x}(e,f)=<A_{x}(e),f>$ is a weak Riemanian metric and
the associated symplectic form $\omega$ is such that $\omega_{u}^{\flat}$ is
an isomorphism for $x\not =e$ and the range of $\omega_{e}^{\flat}$ is only
dense in $T_{e}\mathbb{H}$ (\textit{cf.} \cite{Ma2}).\newline

Let $(e_{k})\in\mathbb{H}$ be a sequence which converges to
$0$. For each $n>0$, we provide $\mathbb{H}_{n}%
=\mathbb{H}\oplus\mathbb{R}^{n}$ of the the inner product  $<\;,\;>_{n}$  obtained from the inner product $<\;,\;>$ on $\mathbb{H}$ and
the canonical one on $\mathbb{R}^{n}$. Now, we consider the continuous 
operator $S_{n}=S\oplus Id_{\mathbb{R}^{n}}$ and, for any $x_{n}\in
\mathbb{H}_{n}$, we set
\[
(A_{n})_{x_{n}}=||x_{n}-e_{n}||^{2}Id_{\mathbb{H}_{n}}+S_{n}%
\]
We denote by $g_{n}$ the Riemannian metric on $\mathbb{H}_{n}$ defined by
\[
(g_{n})_{x_{n}}(e_{n},f_{n})=<(A_{n})_{x}(e_{n}),f_{n}>
\]
and we consider the symplectic form $\omega_{n}$ associated to $g_{n}$ as
previously. 
We set
$\mathcal{H}=\underrightarrow{\lim}\mathbb{H}_{n}$   and 
$M_n=\left(\mathcal{H}\setminus\bigcup\limits_{n\in\mathbb{N}^{\ast}}\{e_{n}\}\right)\bigcap \mathbb{H}_n$.
Now,  the sequence
$\left(  \omega_{n}\right)  _{n\in\mathbb{N}^{\ast}}$ induces a family of coherent
strong symplectic forms on $M_n$ which induces a weak symplectic form $\omega$ on $M=\underrightarrow{\lim}M_n$  since the cotangent space
$T_{x}^{\ast}\mathcal{H}$ is the projective limit of $T_{x_{n}}^{\ast}M_{n}$ and so is then a Fr\'echet space which implies that  $\omega^\flat$ can not be surjective.
Now, for each $n$, we have a Darboux chart $(V_{n},F_{n})$ around $0$. But
from the previous construction, we must have $\bigcap\limits_{n\in
\mathbb{N}^{\ast}}V_{n}=\{0\}$. Therefore there is \textbf{no Darboux chart}
for $\omega$ around $0\in M$ according to Theorem
\ref{equivDarbouxascending}. Note that since each $\omega_{n}$ is a strong
symplectic form, the assumptions of Proposition \ref{Darbouxascending} are satisfied.
\end{example}


\section{A symplectic structure on Sobolev manifolds of loops}

\label{symplecticloops}

\subsection{ Sobolev spaces $\mathsf{L}_{k}^{p}(U,\mathbb{R}^{n})$}${}$

\label{LkpRm}
\textit{In this section, we use the notations of \cite{Sch} and the most part
of this summary comes from \cite{Sch} and \cite{AdFo}}.\newline

We denote by $\mathcal{L}^{k}(\mathbb{R}^{n},\mathbb{R}^{m})$, the vector
space of $k$-multilinear maps from $(R^{n})^{k}$ to $\mathbb{R}^{m}$. This
space is provided with the inner product
\[
A.B=Tr(B^{\ast}A)=\sum_{i=1}^{nk}A(e_{i}).B(e_{i})
\]
where $\left(  e_{1},\dots,e_{nk}\right)  $ is an orthogonal basis of
$(\mathbb{R}^{n})^{k}$ and $B^{\ast}$ is the adjoint of $B$ considered as a
linear map from $\mathbb{R}^{nk}$ to $\mathbb{R}^{m}$. The associated norm is
$||A||=\sqrt{Tr(A^{\ast}A)}$.\newline

Let $U$ be a bounded open set in $\mathbb{R}^{n}$. For $k\in\mathbb{N}%
\cup\left\{  {\infty}\right\}  $, we denote by $C^{k}(U,\mathbb{R}^{m})$ the
set of maps of class $C^{k}$ from $U$ to $\mathbb{R}^{m}$. For $f\in
C^{\infty}(U,\mathbb{R}^{m})$ and for any $\in\mathbb{N}$, we have the
continuous (total) derivative $D^{k}f:U\rightarrow\mathcal{L}^{k}%
(\mathbb{R}^{n},\mathbb{R}^{m})$. On the vector space $C^{k}(U,\mathbb{R}%
^{m})$, for $0\leq k<\infty$ and $1\leq p<\infty$, we consider the two norms:
\[
|f|_{k}=\sum_{i=0}^{k}\sup_{x\in U}||D^{i}f(x)||
\]

\[
||f||_{k,p}=\left(  \sum_{i=0}^{k}\int_{U}||D^{i}(x)f||^{p}\right)  ^{1/p}.
\]

Following \cite{Sch},
we denote by

$\bullet\;\;$$C_{b}(U,\mathbb{R}^{m}):=\{f\in C^{k}(U,\mathbb{R}^{m}%
)\;:\;|u|_{k}<\infty\}$ (which is a Banach space)

$\bullet\;\;$ $\mathsf{L}_{k}^{p}(U,\mathbb{R}^{m})$ the completion of the vector
space
\[
\{f\in C^{k}(U,\mathbb{R}^{m})\}\;:\;||f||_{k,p}<\infty\}.
\]

The collection of all spaces $\mathsf{L}_{k}^{p}(U,\mathbb{R}^{m})$ is
collectively called \textit{Sobolev spaces}. They have an equivalent
formulation as spaces of $p$-integrable functions with $p$-integrable
distributional (or weak) derivatives up to order $k$ (see for instance
\cite{Rou} p 15). \newline 
Now, for all $1<p<\infty$, each space $\mathsf{L}_{k}^{p}(U,\mathbb{R}^{m})$ is a \textit{separable reflexive Banach
space} (see \cite{Rou} p 15). \newline

We have the following embeddings (\textit{cf.} \cite{Bel})%

\[
\mathsf{L}^{p}(U)\subset\mathsf{L}^{q}(U)\; \text{ for } p\geq q;
\]
\[
\mathsf{L}^{p}_{k}(U)\subset\mathsf{L}_{h}^{q}(U)\; \text{ for } k\geq
h,\;p\geq q.
\]
If $U$ is "sufficiently regular":%
\[
\mathsf{L}_{k}^{p}(U)\text{ is dense in }\mathsf{L}_{k-1}^{p}(U)\;\text{ for
}1\leq p<\infty,
\]
and so, by induction
\[
\mathsf{L}_{k}^{p}(U)\text{ is dense in }\mathsf{L}^{p}(U)\;\text{ for }1\leq
p<\infty.
\]

Now, following \cite{AdFo} section 3.5,
for any given integers $n>1$ and $k>0$, let $N=N(n,k)$ be the number of
multi-indices $i=(i_{1},\dots,i_{n})$ such that
\[
|i|=\sum_{\alpha=1}^{n}|i_{\alpha}|\leq k.
\]
For each such multi-index $i$, let $U_{i}$ be a copy of $U$ in a different
copy of $\mathbb{R}^{n}$ so that the $N$ domains $U_{i}$ are considered as
disjoints sets. Let $U^{(k)}$ be the union of these $N$ disjoint domains
$U_{i}$ . Given a function $f\in\mathsf{L}(U,\mathbb{R}^{m})$, let $\tilde
{f}^{(k)}$ be the map on $U^{(k)}$ that coincides with $D^{i}f$ in the
distributional sense. Then the map $P$ taking $f$ to $\tilde{f}^{(k)}$ is an
isometry from $\mathsf{L}_{k}^{p}(U,\mathbb{R})$ into $\mathsf{L}^{p}%
(U^{(k)},\mathbb{R})$ whose range is closed. Now since $\mathsf{L}_{k}%
^{p}(U,\mathbb{R}^{m})=\{(f_{1},\dots,f_{m})\;:\;f_{1},\dots,f_{m}%
\in\mathsf{L}^{p}(U^{(k)},\mathbb{R})\}$, we obtain an isometry $P$ from
$\mathsf{L}_{k}^{p}(U,\mathbb{R}^{m})$ into $\mathsf{L}^{p}(U^{(k)}%
,\mathbb{R})^{m}$ whose range is also closed.\newline

Let $q$ be the "conjugate" exponent to $p$ that is $\dfrac{1}{p}+\dfrac{1}%
{q}=1$. According to \cite{AdFo} section 3.8
we have a "bracket duality" between $\mathsf{L}_{k}^{q}(U,\mathbb{R}^{n})$ and
$\mathsf{L}_{k}^{p}(U,\mathbb{R}^{m})$ given by:
\[
<f,g>=\sum_{j=1}^{m}\int_{U}f_{j}(x)g_{j}(x)dx
\]
From the definition of $\mathsf{L}_{k}^{p}(U,\mathbb{R}^{n})$ and the previous
isometry, we have (\cite{AdFo}):

\begin{theorem}
\label{dualLp} For every $L\in(\mathsf{L}_{k}^{p}(U,\mathbb{R}^{m}))^{\ast}$
there exists $g\in\mathsf{L}_{k}^{q}(U^{(k)},\mathbb{R}^{m})$ such that, if
$g_{i}$ denotes the restriction of $g$ to $U_{i}$, then for all $f\in
\mathsf{L}_{k}^{p}(U,\mathbb{R}^{m})$, we have
\[
L(f)=\sum_{0\leq|i|\leq k}<D^{i}f,g_{i}>
\]

\end{theorem}

In fact on ${L}_{k}^{q}(U,\mathbb{R}^{m})$, we can define the norm
(\textit{cf.} \cite{AdFo} section 3.14)
\[
||g||_{-k,q}=\sup\{|<f,g>|\;:f\in\mathsf{L}_{k}^{p}(U,\mathbb{R}%
^{m}),||f||_{k,p}\leq1\}.
\]
And then $(\mathsf{L}_{k}^{p}(U,\mathbb{R}^{m})^{\ast}$ is the completion of
${L}_{k}^{q}(U,\mathbb{R}^{n})$ according to the previous norm.

\begin{remark}
\label{S1} \normalfont Using the same arguments as in \cite{Sch} Corollary 4.3.3, we obtain
that all the previous results are true by replacing $U$ by a compact manifold
without boundary.
\end{remark}


\subsection{Sobolev manifold structure on the set of loops of a manifold}${}$


Let $M$ be a finite dimensional manifold of dimension $m$. We denote by
$C^{0}(\mathbb{S}^{1},M)$ the set of $C^{0}$-loops of $M$. Now from Theorem
4.3.5 and Theorem 4.4.3 in \cite{Sch}, it follows that there exists a well
defined subset $\mathsf{L}_{k}^{p}(\mathbb{S}^{1},M)$ of $C^{0}(\mathbb{S}%
^{1},M)$ which has a Banach structure modelled on the Banach space
$\mathsf{L}_{k}^{p}(\mathbb{S}^{1},\gamma^{\ast}(TM))$ of sections of "class
$\mathsf{L}_{k}^{p}$" of the pull-back $\gamma^{\ast}(TM)$ over $\mathbb{S}%
^{1}$ for any $\gamma\in C^{\infty}(\mathbb{S}^{1},M)$.
 From  Theorem 4.4.3 in \cite{Sch},  we get the following Theorem:

\begin{theorem}
\label{Lkp} For $k>0$ and $1\leq p<\infty$, there exists a subset
$\mathsf{L}_{k}^{p}(\mathbb{S}^{1},M)$ of the set $C^{0}(\mathbb{S}^{1},M)$
of continuous loops in $M$ which has a Banach manifold structure modelled on
$\mathsf{L}_{k}^{p}(\mathbb{S}^{1}, \mathbb{R}^{m})$ which is a reflexive
Banach space. Moreover the topology of this manifold is  Hausdorff and paracompact.\newline
\end{theorem}

\begin{proof}
\textit{summarized.} For a complete proof see \cite{Sch}.\newline According to
the proof of Theorem 4.4.3 in \cite{Sch}, we only describe how the set
$\mathsf{L}_{k}^{p}(\mathbb{S}^{1},M)$ is built and give an atlas of this
Banach structure since these results will be used latter.\newline Choose any
Riemannian metric on $M$ and denote by $\exp$ the exponential map of its
Levi-Civita connection. Then $\exp$ is defined on an neighbourhood
$\mathcal{U}$ in $TM$ of the zero section. In fact, since $\exp$ is a local
diffeomorphism, if $\pi_{M}:TM\rightarrow M$ is the projection, after
shrinking $\mathcal{U}$ if necessary, we can assume that the map
$\mathcal{F}:=(\exp,\pi_{M})$ is a diffeomorphism from $\mathcal{U}$ onto an
open neighbourhood $\mathcal{V}$ of the diagonal of $M\times M$. Consider any
$\gamma:\mathbb{S}^{1}\rightarrow M$ which is smooth. We consider :
$\mathcal{O}_{\gamma}=\{\gamma^{\prime}\in C^{0}(\mathbb{S}^{1},M)\;:\forall
t\in\mathbb{S}^{1},\;(\gamma(t),\gamma^{\prime}(t))\in\mathcal{V}\}$ and
$\Phi_{\gamma}(\gamma^{\prime})=\mathcal{F}^{-1}(\gamma,\gamma^{\prime})\text{
for }\hat{E}\gamma^{\prime}\in\mathcal{O}_{\gamma}$ \noindent Of course
$\gamma$ belongs to $\mathcal{O}_{\gamma}$. Consider the map $i_{\gamma
}(\gamma^{\prime}):=(\gamma,\gamma^{\prime})$ from $\mathcal{O}$ to
$C^{0}(\mathbb{S}^{1},M\times M)$ then we have $\Phi_{\gamma}=\mathcal{F}%
^{-1}\circ i_{\gamma}$ and so $\Phi_{\gamma}(\mathcal{O}_{\gamma})$ is an open
set in the set $\Gamma^{0}(\gamma^{\ast}TM)\equiv C^{0}(\mathbb{S}%
^{1},\mathbb{R}^{m})$ of continuous sections of $\gamma^{\ast}TM$. We set
\[
U_{\gamma}=\{\gamma^{\prime}\in\mathcal{O}_{\gamma},\gamma^{\prime}%
\in\mathsf{L}_{k}^{p}(\mathbb{S}^{1},\mathbb{R}^{m})\cap\Phi_{\gamma
}(\mathcal{O}_{\gamma})\}
\]
\noindent Now for $k>0$, $\mathsf{L}_{k}^{p}(\mathbb{S}^{1},\mathbb{R}^{m})$
is continuously embedded in $C^{0}(\mathbb{S}^{1},\mathbb{R}^{m})$
(\textit{cf.} consequence of Theorem 4.3.5 in \cite{Sch}) and $C^{\infty
}(\mathbb{S}^{1},\mathbb{R}^{m})\cap\mathsf{L}_{k}^{p}(\mathbb{S}%
^{1},\mathbb{R}^{m})$ is dense in $\mathsf{L}_{k}^{p}(\mathbb{S}%
^{1},\mathbb{R}^{m})$ (consequence of Theorem of Meyers-Serrin \cite{MeSe});
 It follows that $U_{\gamma}$ is a non empty open set of $\mathsf{L}_{k}%
^{p}(\mathbb{S}^{1},\mathbb{R}^{m})$ and again $\gamma$ belongs to $U_{\gamma
}$. Then we set
\[
\mathsf{L}_{k}^{p}(\mathbb{S}^{1},M)=\bigcup_{\gamma\in C^{\infty}%
(\mathbb{S}^{1},M)}U_{\gamma}.
\]
\noindent Then $\{(U_{\gamma},\phi_{\gamma})\}_{\gamma\in C^{\infty
}(\mathbb{S}^{1},M)}$ defines an atlas for a Banach structure modelled on
$\mathsf{L}_{k}^{p}(\mathbb{S}^{1},\mathbb{R}^{m})$.

The proof that   the topology of this manifold is Hausdorff and paracompact in our context is, point by point, analogous to the proof of Corollary 3.23 of \cite{Sta}.\newline
\end{proof}


\subsection{ A symplectic form on $\mathsf{L}_{k}^{p}(\mathbb{S}^{1},M)$}${}$


Let $(M, \omega)$ be a symplectic manifold of dimension $m=2m^{\prime}$. As in
\cite{Ku2}, for any $\gamma\in\mathsf{L}_{k}^{p}(\mathbb{S}^{1},M)$ and $X,
Y\in T_{\gamma}\mathsf{L}_{k}^{p}(\mathbb{S}^{1},\mathbb{R}^{m})\equiv
\mathsf{L}_{k}^{p}(\gamma^{*}TM)$, we set
\begin{align}
\label{defiO}\Omega_{\gamma}(X,Y)=\int_{\mathbb{S}^{1}}\omega_{\gamma
(t)}(X(t),Y(t)) dt
\end{align}

\begin{theorem}
\label{formO} For $k>0$ and $1\leq p\leq\infty$, $\Omega$ is a symplectic
weakly non degenerate form. Moreover $\Omega$ is a strong symplectic form if
and only if $p=2$
\end{theorem}

\begin{proof}
At first, it is clear that $\Omega_{\gamma}$ is a bilinear skew symmetric
$2$-form on $T_{\gamma}\mathsf{L}_{k}^{p}(\mathbb{S}^{1},\mathbb{R}^{m})$.
Since the smoothness of $\gamma\mapsto\Omega_{\gamma}$ is a local property, we
consider a chart $(U,\Phi)$ around $\gamma$ on $\mathsf{L}_{k}^{p}%
(\mathbb{S}^{1},M)$. Note that over $U$, the Banach bundle of bilinear skew
symmetric forms is trivial. Therefore, without loss of generality, we may
assume that $U$ is an open set of $\mathsf{L}_{k}^{p}(\mathbb{S}%
^{1},\mathbb{R}^{m})$ and the previous bundle is
\[
U\times\Lambda^{2}(\mathsf{L}_{k}^{p}(\mathbb{S}^{1},\mathbb{R}^{m})^{\ast
}\equiv U\times\mathsf{L}_{k}^{p}(\mathbb{S}^{1},\Lambda^{2}\mathbb{R}^{m}).
\]
Then the set of smooth section of this bundle can be identified with smooth
map from $U$ to $\mathsf{L}_{k}^{p}(\mathbb{S}^{1},\Lambda^{2}\mathbb{R}%
^{m}))$. First, we prove that the map $(\gamma,X,Y)\mapsto\Omega_{\gamma
}(X,Y)$ is a smooth map from $U\times(\mathsf{L}_{k}^{p}(\mathbb{S}%
^{1},\mathbb{R}^{m}))^{2}$ into $\mathbb{R}$. From Corollary 4.13 in
\cite{KrMi}, it is sufficient to prove that for any smooth curve
$\delta:\mathbb{R}\rightarrow U\times(\mathsf{L}_{k}^{p}(\mathbb{S}^{1}%
,R^{m}))^{2}$ the map
\[
\tau\mapsto\int_{\mathbb{S}^{1}}\omega_{\delta_{0}(\tau)(t)}(X_{\delta
_{1}(\tau)(t)},Y_{\delta_{2}(\tau)(t)})dt
\]
is a smooth map from $\mathbb{R}$ to $\mathbb{R}$, where we have the
decomposition
\[
\delta(\tau)=(\delta_{0}(\tau),\delta_{1}(\tau),\delta_{2}(\tau))\in
U\times(\mathsf{L}_{k}^{p}(\mathbb{S}^{1},\mathbb{R}^{m}))^{2}.
\]
Indeed, for any $t\in\mathbb{S}^{1}$, the map $\tau\mapsto\omega_{\delta
(\tau)(t)}(X_{\delta(\tau)(t)},Y_{\delta(\tau)(t)})$ is smooth and so from the
properties of the integral of a function depending of a parameter, we obtain
the smoothness of $\tau\mapsto\int_{\mathbb{S}^{1}}\omega_{\delta(\tau
)(t)}(X_{\delta(\tau)(t)},Y_{\delta(\tau)(t)})dt$. This implies the smoothness
of $\gamma\mapsto\Omega_{\gamma}$ from \cite{KrMi}, theorem 3.12.\newline
Now we show that $\Omega$ is closed. From Cartan formulae we have
\[
d\Omega(U_{0},U_{1},U_{2})=\sum_{j=0}^{2}(-1)^{j}U_{j}\left(  \Omega
(U_{0},\widehat{U}_{j},U_{2})\right)  +\sum_{0\leq l<j\leq2}\Omega
([U_{l},U_{j}],U_{0},\widehat{U}_{j},U_{2}).
\]
Fix some $\gamma\in\mathsf{L}_{k}^{p}(\mathbb{S}^{1},\mathbb{R}^{m})$.
Consider a map $\sigma:(]-\varepsilon,\varepsilon\lbrack)^{3}\times
\mathbb{S}^{1}\rightarrow M$ with the following properties:

\begin{description}
\item[(i)] $\sigma(0,0,0,t)=\gamma(t)$;

\item[(ii)] $\forall j\in\left\{  0,1,2\right\}  ,\dfrac{\partial\sigma
}{\partial u_{j}}(0,0,0,t)=U_{j}(t)$.
\end{description}

We set $\gamma^{\prime}=\dfrac{\partial\sigma}{\partial t}$ and $U_{j}%
=\dfrac{\partial\sigma}{\partial u_{j}}$, for $j=0,1,2$. Note that by
construction, we have
\[
\forall j\in\left\{  0,1,2\right\}  ,\ [U_{l},U_{j}]=0\;\;.
\]
Now we have $U_{0}\{\Omega(U_{1},U_{2})\}=\dfrac{\partial\sigma}{\partial
u_{0}}\int_{\mathbb{S}^{1}}\omega_{\sigma(0,0,0,t)}(\dfrac{\partial\sigma
}{\partial u_{1}}(0,u_{1},0,t),\dfrac{\partial\sigma}{\partial u_{2}%
}(0,0,u_{2},t))dt$ $\;\;\;\;\;\;\;\;\;\;\;\;\;\;\;\;\;\;\;\;\;\;\;\;=\int
_{\mathbb{S}^{1}}\dfrac{\partial\sigma}{\partial u_{0}}\{\omega_{\sigma
(0,0,0,t)}(\dfrac{\partial\sigma}{\partial u_{1}}(0,u_{1},0,t),\dfrac
{\partial\sigma}{\partial u_{2}}(0,0,u_{2},t))\}dt.$\newline 
Using the same calculus for the other terms in the first sum, we obtain
\[
d\Omega(U_{0},U_{1},U_{2})=\int_{\mathbb{S}^{1}}d\omega_{\gamma(t)}%
(U_{0}(t),U_{1}(t),U_{2})(t)dt.
\]
Therefore, since $\omega$ is closed so is $\Omega$.\newline It remains to show
that $\Omega$ is weakly non degenerate. Assume that there exists $X\in
T_{\gamma}\mathsf{L}_{k}^{p}(\mathbb{S}^{1},\mathbb{R}^{m})$ such that for any
$Y\in T_{\gamma}\mathsf{L}_{k}^{p}(\mathbb{S}^{1},\mathbb{R}^{m})$ we have
$\Omega_{\gamma}(X,Y)=0$.\newline Since $k>0$, then we can assume that
$\gamma$ and $X$ are continuous (\textit{cf.} consequence of Theorem 4.1.1 in
\cite{Sch}). Now $X$ is not identically zero, so there exists some interval
$]\alpha,\beta\lbrack\subset\mathbb{S}^{1}$ on which $\omega_{\gamma
(t)}(X(t),\;)$ is a field of non zero $1$-forms. Let $J\subsetneqq
]\alpha,\beta\lbrack$ be a closed sub-interval. Since $\omega$ is symplectic,
it must exist a field of smooth vector fields $Y_{J}$ on $J$ such that
$\omega_{\gamma(t)}(X(t),Y_{j}(t))>0$. We can extend $Y_{J}$ to a smooth
vector field $Y$ on $\mathbb{S}^{1}$ so that the support of $Y$ is contained
in $]\alpha,\beta\lbrack$. Since $Y$ is smooth then $Y$ belongs to $T_{\gamma
}\mathsf{L}_{k}^{p}(\mathbb{S}^{1},\mathbb{R}^{m})$ and, by construction, we
have
\[
\Omega(X,Y)=\int_{\mathbb{S}^{1}}\omega_{\gamma(t)}(X(t),Y(t))dt>0.
\]
and we obtain a contradiction.\newline Now, if $\Omega$ strong symplectic
form, this implies that the Banach space $\mathsf{L}_{k}^{p}(\mathbb{S}%
^{1},\mathbb{R}^{m})$ is isomorphic to its dual which is only true for $p=2$.
Conversely, for $p=2$, the space $\mathsf{L}_{k}^{2}(\mathbb{S}^{1}%
,\mathbb{R}^{m})$ is a Hilbert space and so $T_{\gamma}^{\ast}\mathsf{L}%
_{k}^{2}(\mathbb{S}^{1},M)$ is isomorphic to $\mathsf{L}_{k}^{2}%
(\mathbb{S}^{1},\mathbb{R}^{m})$. Since $\omega$ is symplectic, $\omega
^{\flat}$ is a smooth isomorphism from $TM$ to $T^{\ast}M$. Therefore, if
$\eta\in T_{\gamma}^{\ast}\mathsf{L}(\mathbb{S}^{1},M)$, $X=(\omega^{\flat
})^{-1}(\eta)$ belongs to $T_{\gamma}\mathsf{L}_{k}^{2}(\mathbb{S}^{1},M)$
and, for any $Y\in T_{\gamma}\mathsf{L}_{k}^{2}(\mathbb{S}^{1},M),$ we have
\[
\omega_{\gamma(t)}(X(t),Y(t))=<\eta(t),Y(t)>_{M}%
\]
where $<\;,\;>_{M}$ is the duality bracket between $TM$ and $T^{\ast}M$. Now
the duality bracket between $T_{\gamma}^{\ast}\mathsf{L}_{k}^{2}%
(\mathbb{S}^{1},M)$ and $T_{\gamma}\mathsf{L}_{k}^{2}(\mathbb{S}^{1},M)$ is
then given by
\[
\int_{\mathbb{S}^{1}}<\xi(t),Z(t)>dt
\]
for any $\xi\in T_{\gamma}^{\ast}\mathsf{L}_{k}^{2}(\mathbb{S}^{1},M)$ and any
$Z\in T_{\gamma}\mathsf{L}_{k}^{2}(\mathbb{S}^{1},M)$. Therefore we have
\[
\Omega^{\flat}(X)(Y)=\eta(Y)
\]
which ends the proof.\newline
\end{proof}


\subsection{Darboux chart on $\mathsf{L}_{k}^{p}(\mathbb{S}^{1},M)$}${}$


We are now in situation to apply Theorem \ref{localDarboux} for the symplectic
form on $\mathsf{L}_{k}^{p}(\mathbb{S}^{1},M)$ as defined in the previous section.

\begin{theorem}
\label{darbouxO} There exists a Darboux chart $(V,F)$ around any $\gamma
\in\mathsf{L}_{k}^{p}(\mathbb{S}^{1},M)$ for the weak symplectic manifold
$(\mathsf{L}_{k}^{p}(\mathbb{S}^{1},M),\Omega)$.
\end{theorem}

\begin{proof}
At first note that if $p=2$, then $\Omega$ is a strong symplectic form and
then the result is an application of the result of \cite{Ma1} and \cite{Wei}
(or also Theorem \ref{localDarboux}). Thus, from now, we assume $1<p<\infty$
and $p\not =2$.\newline Fix some $\gamma\in\mathsf{L}_{k}^{p}(\mathbb{S}%
^{1},M)$ and consider a chart $(U,\Phi)$ around $\gamma$. It is well known
that the maps $T\phi:T\mathsf{L}_{k}^{p}(\mathbb{S}^{1},M)_{|U}\rightarrow
\Phi(U)\times\mathsf{L}_{k}^{p}(\mathbb{S}^{1},\mathbb{R}^{m})$ and $T^{\ast
}\Phi^{-1}:T^{\ast}\mathsf{L}_{k}^{p}(\mathbb{S}^{1},M)_{|U}\rightarrow
\Phi(U)\times(\mathsf{L}_{k}^{p}(\mathbb{S}^{1},\mathbb{R}^{m}))^{\ast}$ are
the natural trivializations associated to $(U,\Phi)$. Without loss of
generality, we may assume that $U$ is an open subset of $\mathsf{L}_{k}%
^{p}(\mathbb{S}^{1},\mathbb{R}^{m})$. Then $TU=U\times\mathsf{L}_{k}%
^{p}(\mathbb{S}^{1},\mathbb{R}^{m})$ and $T^{\ast}U=U\times(\mathsf{L}_{k}%
^{p}(\mathbb{S}^{1},\mathbb{R}^{m}))^{\ast}$. \newline According to Remark
\ref{S1} and property (\ref{Lkp}), it follows that $\mathsf{L}_{k}%
^{p}(\mathbb{S}^{1},\mathbb{R}^{m})$ is dense in $\mathsf{L}^{p}%
(\mathbb{S}^{1},\mathbb{R}^{m})$. Either $1<p<2$ and then $p\leq q$, which
implies $\mathsf{L}^{q}(\mathbb{S}^{1},\mathbb{R}^{m})$ is a dense subspace of
$\mathsf{L}^{p}(\mathbb{S}^{1},\mathbb{R}^{m})$ and so $\mathsf{L}_{k}%
^{p}(\mathbb{S}^{1},\mathbb{R}^{m})\cap\mathsf{L}^{q}(\mathbb{S}%
^{1},\mathbb{R}^{m})$ is dense in $\mathsf{L}^{q}(\mathbb{S}^{1}%
,\mathbb{R}^{m})$. Otherwise, if $p>2$, then $\mathsf{L}^{p}(\mathbb{S}%
^{1},\mathbb{R}^{m})$ is a dense subspace of $\mathsf{L}^{q}(\mathbb{S}%
^{1},\mathbb{R}^{m})$ and so $\mathsf{L}_{k}^{p}(\mathbb{S}^{1},\mathbb{R}%
^{m})$ is dense in $\mathsf{L}^{q}(\mathbb{S}^{1},\mathbb{R}^{m})$. Thus we
will only consider the case $p>2$ since the other case is similar by replacing
$\mathsf{L}_{k}^{p}(\mathbb{S}^{1},\mathbb{R}^{m})$ by $\mathsf{L}_{k}%
^{p}(\mathbb{S}^{1},\mathbb{R}^{m})\cap\mathsf{L}^{q}(\mathbb{S}%
^{1},\mathbb{R}^{m})$. Since $TU=U\times\mathsf{L}_{k}^{p}(\mathbb{S}%
^{1},\mathbb{R}^{m})$, and $\mathsf{L}_{k}^{p}(\mathbb{S}^{1},\mathbb{R}^{m})$
is dense in $\mathsf{L}^{q}(\mathbb{S}^{1},\mathbb{R}^{m})$ we can consider
the trivial bundle
\[
L^{q}(U)=U\times\mathsf{L}^{q}(\mathbb{S}^{1},\mathbb{R}^{m}),
\]
and then the inclusion of $TU$ in $L^{q}(U)$ is a bundle morphism with dense
range.\newline 
On the other hand, according to the end of section \ref{LkpRm}
and Remark \ref{S1}, $\mathsf{L}^{q}(\mathbb{S}^{1},\mathbb{R}^{m})$ provided
with the norm $||\;||_{-k,p}$ is dense in
$(\mathsf{L}_{k}^{p}(\mathbb{S}^{1},\mathbb{R}^{m}))^{\ast}$. So we have an
inclusion of the trivial bundle $L^{q}(U)$ in $T^{\ast}U$ whose range is
dense.\newline Since $\omega$ is symplectic, $\omega_{\gamma}^{\flat}%
:\gamma^{\ast}(TM)\rightarrow\gamma^{\ast}(T^{\ast}M)$ is an isomorphism and
the map $X\in\mathsf{L}_{k}^{p}(\mathbb{S}^{1},\mathbb{R}^{m})\mapsto
\alpha=\omega_{\gamma(t)}^{\flat}(X(t)$ is nothing but that $\Omega_{\gamma
}^{\flat}$. For each fixed $\gamma$, according to the end of section
\ref{linearsymplectic}, we can extend the operator ${\Omega}_{\gamma}^{\flat
}:\mathsf{L}_{k}^{q}(\mathbb{S}^{1},\mathbb{R}^{m})\rightarrow(\mathsf{L}%
_{k}^{p}(\mathbb{S}^{1},\mathbb{R}^{m}))^{\ast}$ to a continuous operator
$(\bar{\Omega}_{\gamma}^{\flat})^{q}:\mathsf{L}^{q}(\mathbb{S}^{1}%
,\mathbb{R}^{m})\rightarrow(\mathsf{L}_{k}^{p}(\mathbb{S}^{1},\mathbb{R}%
^{m}))^{\ast}$. In fact this operator is given by
\[
\bar{\Omega}_{\gamma}^{\flat}(X(t))=\omega_{\gamma(t)}^{\flat}(X(t))
\]
and $\bar{\Omega}^{\flat}(X)$ belongs to $\mathsf{L}^{q}(\mathbb{S}%
^{1},\mathbb{R}^{m})$. Therefore the range of $\bar{\Omega}_{\gamma}^{\flat}$
is $\mathsf{L}^{q}(\mathbb{S}^{1},\mathbb{R}^{m})\subset(\mathsf{L}_{k}%
^{p}(\mathbb{S}^{1},\mathbb{R}^{m}))^{\ast}$ and is an isomorphism whose
inverse is given by
\[
(\bar{\Omega}^{\flat})^{-1}(\alpha(t))=(\omega_{\gamma(t)}^{\flat}%
)^{-1}(\alpha(t))
\]
Note that $\bar{\Omega}_{\gamma}^{\flat}$ is the natural morphism associated
to the skew symmetric bilinear form on $\mathsf{L}^{q}(\mathbb{S}%
^{1};\mathbb{R}^{m})$ defined by
\[
\bar{\Omega}_{\gamma}(X,Y)=\int_{\mathbb{S}^{1}}\omega_{\gamma(t)}%
(X(t),Y(t)dt
\]
which is an extension of the initial $2$-form $\Omega_{\gamma}$ on
$\mathsf{L}^{p}(\mathbb{S}^{1};\mathbb{R}^{m})$. Therefore, on $\mathsf{L}%
^{q}(\mathbb{S}^{1};\mathbb{R}^{m})$, considered as a subset of $(\mathsf{L}%
_{k}^{p}(\mathbb{S}^{1},\mathbb{R}^{m}))^{\ast}$, we can consider the induced
norm $||\;||_{-k,q}$. Since $\bar{\Omega}_{\gamma}$ is bijective, then
\[
||X||_{\bar{\Omega}_{\gamma}}=||\bar{\Omega}^{\flat}(X)||_{-k,p}%
\]
defines a norm on $\mathsf{L}^{q}(\mathbb{S}^{1};\mathbb{R}^{m})$ for all
$\gamma\in U$. This implies that $\bar{\Omega}_{\gamma}^{\flat}$ is a
continuous isomorphism and an isometry from the normed space $(\mathsf{L}%
^{q}(\mathbb{S}^{1};\mathbb{R}^{m}),||\;||_{\bar{\Omega}_{\gamma}})$ to the
normed space $(\mathsf{L}^{q}(\mathbb{S}^{1};\mathbb{R}^{m}),||\;||_{-k,q})$.
Therefore the norms $||\;||_{\bar{\Omega}_{\gamma}}$ and $||\;||_{-k,q}$ are
equivalent on $(\mathsf{L}^{q}(\mathbb{S}^{1};\mathbb{R}^{m})$ for any
$\gamma\in U$. Now, note that $||\;||_{\bar{\Omega}_{\gamma}}$ induces the
norm $||\;||_{{\Omega}_{\gamma}}$ on $T_{\gamma}U=\{\gamma\}\times
\mathsf{L}_{k}^{p}(\mathbb{S}^{1};\mathbb{R}^{m})$ as defined in section
\ref{DarbouxBanach}. Since each norm $||\;||_{\bar{\Omega}_{\gamma}}$ is
equivalent to $||\;||_{-k,p}$ then each norm $||\;||_{\bar{\Omega}_{\gamma}}$
is equivalent to $||\;||_{\bar{\Omega}_{\gamma_{0}}}$, it follows that the
same result is true for the restriction of this norm to $\mathsf{L}_{k}%
^{p}(\mathbb{S}^{1},\mathbb{R}^{m})$. Therefore, the completion of the normed
space $(\mathsf{L}_{k}^{p}(\mathbb{S}^{1},\mathbb{R}^{m}),||\;||_{\Omega
_{\gamma}})$ does not depend on $\gamma$. We denote by $\widehat
{\mathsf{L}_{k}^{p}}(\mathbb{S}^{1},\mathbb{R}^{m})_{\gamma}$ this completion.
According to the notations of section \ref{DarbouxBanach}, we have
$\widehat{TU}=U\times\widehat{\mathsf{L}_{k}^{p}}(\mathbb{S}^{1}%
,\mathbb{R}^{m})_{\gamma_{0}}$. This implies that the condition (i) of the
assumptions of Theorem \ref{localDarboux} is satisfied.\newline 
As we have seen in section \ref{linearsymplectic}, for each $\gamma$, we can extend the
operator
\[
\Omega^{\flat}:\mathsf{L}_{k}^{p}(\mathbb{S}^{1},\mathbb{R}^{m})\rightarrow
(\mathsf{L}_{k}^{p}(\mathbb{S}^{1},\mathbb{R}^{m}))^{\ast}%
\]
to an isometry $\widehat{\Omega}_{\gamma}^{\flat}$ from $\widehat
{\mathsf{L}_{k}^{p}}(\mathbb{S}^{1},\mathbb{R}^{m})_{\gamma}$ to
$(\mathsf{L}_{k}^{p}(\mathbb{S}^{1},\mathbb{R}^{m}))^{\ast}$. It remains to
show that $\gamma\mapsto{\Omega}_{\gamma}^{\flat}$ is a smooth field from $U$
to $\mathcal{L}((\mathsf{L}_{k}^{p}(\mathbb{S}^{1},\mathbb{R}^{m}%
)),(\widehat{\mathsf{L}_{k}^{p}}(\mathbb{S}^{1},\mathbb{R}^{m})_{\gamma_{0}%
})^{\ast})$.\newline This property can be shown by the same argument used in
Theorem \ref{formO}, for the smoothness of $\gamma\mapsto\Omega_{\gamma}$.
Therefore the assumptions for applying Theorem \ref{localDarboux} are
satisfied on $U$ and this ends the proof.\newline
\end{proof}


\subsection{Moser's method on $\mathsf{L}_{k}^{p}(\mathbb{S}^{1},M)$}${}$


In this section, denote by $I$ the interval $[0,1]$. According to \cite{Ku2},
an \textit{isotopy in a Banach manifold $\mathcal{M}$} is a smooth map
$F:I\times\mathcal{M}\longrightarrow\mathcal{M}$ such that for all $s\in I$
the induced map $F_{s}:\mathcal{M}\longrightarrow\mathcal{M}$ is a
diffeomorphism with $F_{0}=Id$.

Assume that we have an isotopy $F:I\times M\longrightarrow M$. Let ${F}%
^{L}:I\times\mathsf{L}_{k}^{p}(\mathbb{S}^{1},M)\rightarrow\mathsf{L}_{k}%
^{p}(\mathbb{S}^{1},M)$ be the map defined by
\[
{F}^{L}(s,\gamma):t\mapsto F_{s}(\gamma(t)).
\]

Using the same argument as in section 2 of \cite{Ku2}, we have

\begin{proposition}
\label{isotopie} ${F}^{L}$ is an isotopy in $\mathsf{L}_{k}^{p}(\mathbb{S}%
^{1},M)$. If $X\in T_{\gamma}\mathsf{L}_{k}^{p}(\mathbb{S}^{1},M)\equiv
\mathsf{L}_{k}^{p}(\mathbb{S}^{1},\mathbb{R}^{m})$ then the differential of
${F}^{L}_{s}$ is the map $t\mapsto D_{\gamma}F_{s}(X(t))$
\end{proposition}

\begin{proof}Consider a smooth curve $c:\mathbb{R}\rightarrow I\times\mathsf{L}_{k}%
^{p}(\mathbb{S}^{1},M)$ which is denoted $(c_{1}(\tau),c_{2}(\tau))$. Consider
the curve $\tilde{c}:\mathbb{R}\rightarrow\mathsf{L}_{k}^{p}(\mathbb{S}%
^{1},M)$ given by
\[
\tilde{c}(\tau)=F^{L}\circ c(\tau)=F_{c_{1}(\tau)}(c_{2}(\tau)(t))
\]
According to the proof of Theorem \ref{Lkp}, there exists a chart $(U,\Phi)$
such that $\tilde{c}$ belongs to $U$ and then $\Phi(\tilde{c})$ belongs to
$\mathsf{L}_{k}^{p}(\mathbb{S}^{1},\mathbb{R}^{m})$. Therefore $\tilde{c}$ is
smooth if and only if $\Phi\circ\tilde{c}:\mathbb{R}\rightarrow\mathsf{L}%
_{k}^{p}(\mathbb{S}^{1},\mathbb{R}^{m})$ is smooth. Now a curve $\tau
\mapsto\delta$ in the Banach space $\mathsf{L}_{k}^{p}(\mathbb{S}%
^{1},\mathbb{R}^{m})$ is smooth if for any $\lambda\in(\mathsf{L}_{k}%
^{p}(\mathbb{S}^{1},\mathbb{R}^{m}))^{\ast}$ the curve $\tau\mapsto
\lambda(\delta)(\tau)$ is a smooth curve from $\mathbb{R}$ to $\mathbb{R}$
(\textit{cf.} \cite{KrMi}, Theorem 2.14). But from Theorem \ref{dualLp},
$\delta$ will be smooth if for any $g\in\mathsf{L}_{k}^{q}((\mathbb{S}%
^{1})^{(k)},\mathbb{R}^{n})$ the curve $\tau\mapsto\sum\limits_{0\leq|i|\leq
k}<D^{i}\delta,g_{i}>(\tau)$ is smooth.\newline For $\delta(\tau
)=F(\phi_{c_{1}(\tau)}(c_{2}(\tau)(t))$, since $c_{1}$, $c_{2}$, $\phi$ and
$\Phi$ are smooth, according to the definition on the "dual bracket", it follows that, for any $g\in\mathsf{L}_{k}^{q}%
((\mathbb{S}^{1})^{(k)},\mathbb{R}^{n})$, the curve $\tau\mapsto
\sum\limits_{0\leq|i|\leq k}<D^{i}\delta,g_{i}>(\tau)$ is smooth. This implies
that ${F^{L}}$ is smooth (\textit{cf.} \cite{KrMi}). Since $F_{s}$ is a
diffeomorphism, its inverse $G_{s}$ is defined and, in the same way, we can
show that ${G^{L}}$ is smooth. Finally, we obtain that ${F^{L}}$ is an isotopy
in $\mathsf{L}_{k}^{p}(\mathbb{S}^{1},M)$. The proof of the last part is point
by point the same as the proof in Proposition 2.1 of \cite{Ku2}.\newline
\end{proof}

\begin{remark}
\label{diffeo} \normalfont If we consider a smooth diffeomorphism $F$ of a manifold $M$,
by the same arguments as in the proof of Proposition \ref{isotopie}, we can
show that the map $F^{L}:\mathsf{L}_{k}^{p}(\mathbb{S}^{1},\mathbb{R}%
^{m})\rightarrow\mathsf{L}_{k}^{p}(\mathbb{S}^{1},\mathbb{R}^{m})$ defined by
$F^{L}(\gamma)=F\circ\gamma$ is a smooth diffeomorphism of $\mathsf{L}_{k}%
^{p}(\mathbb{S}^{1},\mathbb{R}^{m})$.
\end{remark}

Assume that $\{\omega^{s}\}_{s\in\lbrack0,1]}$ is a $1$-parameter family of
symplectic forms on $M$ which smoothly depends on $s$. To $\omega^{s}$ is
associated a symplectic form $\Omega^{s}$ defined by the relation
(\ref{defiO}). Then $\{\Omega^{s}\}_{s\in\lbrack0,1]}$ is also a $1$-parameter
family of symplectic forms on $\mathsf{L}_{k}^{p}(\mathbb{S}^{1},M)$ which
smoothly depends on $t$. As in \cite{Ku2}, we have

\begin{theorem}
\label{isopdarboux} Just like before, consider the $1$-parameter family of
symplectic forms $\{\omega^{s}\}_{s\in\lbrack0,1]}$ on $M$. Assume that there
exists an isotopy $F:I\times M\longrightarrow M$ such that $F_{s}^{\ast}%
\omega^{s}=\omega^{0}$, then $({F}_{s}^{L})^{\ast}(\Omega^{s})=\Omega^{0}$.
\end{theorem}

Using Proposition \ref{isotopie}, the proof of Theorem \ref{isopdarboux} is
formally the same as the proof of Theorem 4.1 of \cite{Ku2}

\begin{remark}
\label{diffodarboux} \normalfont Let $F$ be a diffeomorphism on $M$ and $\omega_{0}$ and
$\omega_{1}$ be two symplectic forms on $M$ such that $F^{*}\omega_{1}%
=\omega_{0}$. For $i=1,2$, if $\Omega_{i}$ is the symplectic form on
$\mathsf{L}_{k}^{p}(\mathbb{S}^{1},M)$ defined from $\omega_{i}$ by the
relation (\ref{defiO}), using the same arguments as in the proof of Theorem
4.1 of \cite{Ku2}, we show that $(F^{L})^{*}\Omega_{1}=\Omega_{0}$. In
particular, this implies that if $V$ contains $\gamma(I)$ and $(V,\Phi)$ is a
Darboux chart for $\omega$ we obtain by this way a Darboux chart for $\Omega$
in $\mathsf{L}_{k}^{p}(\mathbb{S}^{1},M)$.
\end{remark}






\subsection{ A symplectic form on the direct limit of an ascending sequence of Sobolev manifolds of loops}${}$

\label{symplecasceloop}

We consider an ascending sequence $\{\left(  M_{n},\omega_{n}\right)
\}_{n\in\mathbb{N}^{\ast}}$ of finite dimensional symplectic manifolds such
that $\omega_{n}$ is the restriction of $\omega_{n+1}$ to $M_{n}$ for all
$n>0$. Then, according to the previous section, for each $n$, we denote by
$\Omega_{n}$ the weak symplectic form on $\mathsf{L}_{k}^{p}(\mathbb{S}%
^{1},M_{n})$ associated to $\omega_{n}$. It is clear that $\left(
\mathsf{L}_{k}^{p}(\mathbb{S}^{1},M_{n})\right)  _{n\in\mathbb{N}^{\ast}}$ is
an ascending sequence of Banach manifolds.

\begin{proposition}
\label{symplectic ascendingloop}$M=\underrightarrow{\lim}M_{n}$ is modelled on
the convenient space $\mathbb{R}^{\infty}$.\\
The direct limit $\mathsf{L}_{k}^{p}(\mathbb{S}^{1},M)=\underrightarrow{\lim}(\mathsf{L}_{k}%
^{p}(\mathbb{S}^{1},M_{n}))$ is contained in the set of continuous loops in
$\mathbb{R}^{\infty}$ and has a Hausdorff  convenient manifold structure modelled
on the convenient space $\mathsf{L}_{k}^{p}(\mathbb{S}^{1},\mathbb{R}^{\infty
})=\underrightarrow{\lim}(\mathsf{L}_{k}^{p}(\mathbb{S}^{1},\mathbb{R}^{m}))$.
Moreover, the sequence of symplectic forms $\left(  \Omega_{n}\right)  $ is
coherent and the induced symplectic form $\Omega$ on $\mathsf{L}_{k}^{p}(\mathbb{S}^{1},M)$ is characterized by:
\[
\Omega_{\gamma}(X,Y)=\int_{\mathbb{S}^{1}}\omega_{\gamma}(t)(X(t),Y(t))dt
\]
In fact, $(\mathsf{L}_{k}^{p}(\mathbb{S}^{1},M),\Omega)$ satisfies the
assumption of Proposition \ref{Darbouxascending}.
\end{proposition}

\begin{proof}
At first, from \cite{Glo}, the direct limit $M=\underrightarrow{\lim}M_{n}$ is
a (Hausdorff) convenient manifold modelled on the convenient space
$\mathbb{R}^{\infty}$.\\
Now, if $\gamma$ belongs to $\mathsf{L}_{k}%
^{p}(\mathbb{S}^{1},M)$, then $\gamma$ belongs to some $\mathsf{L}_{k}%
^{p}(\mathbb{S}^{1},M_{n})$ and so is a continuous loop in $M$. The direct
limit $\mathsf{L}_{k}^{p}(\mathbb{S}^{1},\mathbb{R}^{\infty})=\underrightarrow
{\lim}(\mathsf{L}_{k}^{p}(\mathbb{S}^{1},\mathbb{R}^{m}))$ is a convenient
space and we will use Theorem 40 of \cite {CaPe} to show that the direct limit $\mathsf{L}_{k}^{p}%
(\mathbb{S}^{1},M)=\underrightarrow{\lim}(\mathsf{L}_{k}^{p}(\mathbb{S}%
^{1},M_{n}))$ is a convenient manifold modelled on this space. Therefore, we must prove that  $\mathsf{L}_{k}^{p}%
(\mathbb{S}^{1},M)=\underrightarrow{\lim}(\mathsf{L}_{k}^{p}(\mathbb{S}%
^{1},M_{n}))$  has the direct chart limit property. More precisely,
 according to Proposition \ref{P_ConditionsDefiningDLCP}, we must show
that for any $\gamma=\underrightarrow{\lim}\gamma_{n}\in\mathsf{L}_{k}%
^{p}(\mathbb{S}^{1},M)$, there exists an ascending sequence of charts
$\lbrace(U_{n},\phi_{n})\rbrace_{n\in\mathbb{N^\ast}}$ around $\gamma_{n}$ and then $U=\underrightarrow{\lim}%
U_{n},\phi=\underrightarrow{\lim}\phi_{n}$ will be a chart around $\gamma$. We
use the context of the proof of Theorem \ref{Lkp}.\newline 
On each $M_{n}$, we choose a Riemannian metric $g_{n}$ such that the restriction of $g_{n+1}$ to
$M_{n}$ is $g_{n}$. We denote by $\exp_{n}$ the exponential map associated to
the Levi-Civita connection of $g_{n}$. For each $n$, we have an open
neighbourhood $\mathcal{U}_{n}$ of the zero section of $TM_{n}$ such that
$\mathcal{F}_{n}:=(\exp_{n},\pi_{M_{n}})$ is a diffeomorphism from
$\mathcal{U}_{n}$ onto a neighbourhood $\mathcal{V}_{n}$ of the diagonal of
$M_{n}\times M_{n}$. Moreover, since the restriction of $g_{n+1}$ to $M_{n}$
is $g_{n}$, we can choose $\mathcal{U}_{n+1}$ and $\mathcal{U}_{n}$ such that
$\mathcal{U}_{n}$ is contained in $\mathcal{U}_{n+1}$ and the restriction of
$\mathcal{F}_{n+1}$ to $\mathcal{U}_{n}$ is $\mathcal{F}_{n}$. Fix some
$\gamma=\underrightarrow{\lim}\gamma_{n}$ in $M$. Recall that if $\gamma_{n}$
belongs to $\mathsf{L}_{k}^{p}(\mathbb{S}^{1},M_{n})$, then $\gamma
_{n+1}=\gamma_{n}$ and so $\gamma_{n}$ belongs to $\mathsf{L}_{k}%
^{p}(\mathbb{S}^{1},M_{n+1})$. Now, from the construction of charts on
$\mathsf{L}_{k}^{p}(\mathbb{S}^{1},M_{n})$, it is clear that we have a chart
$(U_{n},\phi_{n})$ in $\mathsf{L}_{k}^{p}(\mathbb{S}^{1},M_{n})$ which
contains $\gamma_{n}$ and a chart $(U_{n+1},\phi_{n+1})$ in $\mathsf{L}%
_{k}^{p}(\mathbb{S}^{1},M_{n})$ such that $U_{n}\subset U_{n+1}$ and the
restriction of $\phi_{n+1}$ to $U_{n}$ is $\phi_{n}. $ Finally, since each Banach manifold $\mathsf{L}_{k}^{p}(\mathbb{S}
^{1},M_{n}))$  is Hausdorff and paracompact, from Proposition 15 of \cite{CaPe},  we get that  $\mathsf{L}_{k}^{p}(\mathbb{S}^{1},M)$
has a  structure of Hausdorff  convenient manifold. \newline 
Recall that for each $\gamma\in\mathsf{L}_{k}%
^{p}(\mathbb{S}^{1},M_{n})$ we have
\[
\Omega_{\gamma}(X,Y)=\int_{\mathbb{S}^{1}}(\omega_{n})_{\gamma(t)}%
(X(t),Y(t))dt.
\]
Since $\omega_{n}$ is the restriction of $\omega_{n+1}$ to $M_{n}$, this
implies that the restriction of $\Omega_{n+1}$ to $\mathsf{L}_{k}%
^{p}(\mathbb{S}^{1},M_{n})$ is $\Omega_{n}$. Therefore the sequence $\left(
\Omega_{n}\right)  $ is coherent and so defines a symplectic form $\Omega$ on
$\mathsf{L}_{k}^{p}(\mathbb{S}^{1},M)$ (\textit{cf.} Theorem
\ref{omega_ncoherentM}). Now, from the same argument, the sequence $\left(
\omega_{n}\right)  $ defines a symplectic form $\omega$ on $M$. Let $\gamma
\in\mathsf{L}_{k}^{p}(\mathbb{S}^{1},M)$ and $X,Y$ in $T_{\gamma}%
\mathsf{L}_{k}^{p}(\mathbb{S}^{1},M)$. There exists $n$ such that $\gamma$
belongs to $\mathsf{L}_{k}^{p}(\mathbb{S}^{1},M_{n})$ and $X,Y$ belongs to
$T_{\gamma}\mathsf{L}_{k}^{p}(\mathbb{S}^{1},M_{n})$. The definitions of
$\Omega$ and of $\omega$ imply that:
\[
\Omega_{\gamma}(X,Y)=\Omega_{n}(X,Y)=\int_{\mathbb{S}^{1}}(\omega_{n}%
)_{\gamma}(t)(X(t),Y(t))dt=\int_{\mathbb{S}^{1}}\omega_{\gamma}%
(t)(X(t),Y(t))dt
\]
since in this case, $(\omega_{n})_{\gamma(t)}(X(t),Y(t))=\omega_{\gamma
(t)}(X(t),Y(t))$. Moreover, this relation is independent of the choice of such
$n$. Now according to the proof of Theorem \ref{darbouxO}, the assumptions of
Proposition \ref{Darbouxascending} are satisfied.
\end{proof}

Unfortunately, even in the particular case of direct limit of Sobolev
manifolds of loops, while all the assumptions of Propositon
\ref{Darbouxascending} are satisfied, we have no general result about the existence of a
Darboux chart in this situation. However, we have at least the following
Example of direct limit of Sobolev manifolds of loops for which there exists
a Darboux chart for $\Omega$ at any point:

\begin{example}\normalfont
\label{exDarbouxascending} 
We consider $\mathbb{R}^{\infty}=\underrightarrow
{\lim}\mathbb{R}^{2n}$. Let $\left(  e_{1},\dots,e_{2n}\right)  $ be the
canonical basis of $\mathbb{R}^{2n}$ such that $\left(  e_{1},\dots
,e_{2n-2}\right)  $ is the canonical basis of $\mathbb{R}^{2n-2}$ for all
$n>0$. Any global diffeomorphism $\phi_{n}$ of $\mathbb{R}^{2n}$ defines a
chart of the natural manifold structure on $\mathbb{R}^{2n}$. Given the
canonical Darboux form
\[
\eta_{n}=\sum_{i=1}^{n}dx_{i}\wedge dy_{i}%
\]
then $\omega_{n}=\phi_{n}^{\ast}\eta_{n}$ is also a symplectic form and
$(\mathbb{R}^{2n},\psi_{n}=\phi_{n}^{-1})$ is a global Darboux chart for
$\omega_{n}$. We now consider a family $\{\phi_{n}\}$ such that $\phi_{n}$ is
a diffeomorphism of $\mathbb{R}^{2n}$ and the restriction of $\phi_{n}$ to
$\mathbb{R}^{2n-2}$ is $\phi_{n-1}$ for all $n>1$. We set $\omega_{n}=\phi
_{n}^{\ast}\eta_{n}$. Then $\left(  \omega_{n}\right)  $ is a sequence of
coherent symplectic forms and so defines a weak symplectic form $\omega$ on
$\mathbb{R}^{\infty}$. Moreover, if $\psi=\underrightarrow{\lim}\psi_{n}$,
from Theorem \ref{equivDarbouxascending}, $(\mathbb{R}^{\infty},\psi)$ is a
global Darboux chart. \newline 
On the other hand, according to Proposition \ref{symplectic ascendingloop}, we have a symplectic form $\Omega_{n}$ on each manifold $\mathsf{L}_{k}^{p}(\mathbb{S}^{1},\mathbb{R}^{2n})$ and also a
symplectic form $\Omega$ on $\mathsf{L}_{k}^{p}(\mathbb{S}^{1},\mathbb{R}%
^{\infty})$. Since $(\mathbb{R}^{2n},\psi_{n})$ is a (global) Darboux chart
for $\omega_{n}$, then, from Remark \ref{diffeo} , Remark \ref{diffodarboux}
and Theorem \ref{omega_ncoherentM}, if $\psi_{n}^{L}$ is the diffeomorphism of
$(\mathsf{L}_{k}^{p}(\mathbb{S}^{1},\mathbb{R}^{2n})$ induced by $\psi_{n}$,
then $(\mathsf{L}_{k}^{p}(\mathbb{S}^{1},\mathbb{R}^{2n}),\psi_{n}^{L})$ is
also a (global) Darboux chart for $\Omega_{n}$. Finally, let $\psi
^{L}=\underrightarrow{\lim}\psi_{n}^{L}$. Then from Theorem
\ref{equivDarbouxascending} again, $(\mathsf{L}_{k}^{p}(\mathbb{S}%
^{1},\mathbb{R}^{\infty}),\Psi^{L})$ is a (global) Darboux chart for $\Omega$.
\newline
\end{example}

\appendix

\section{Direct limit of ascending sequences of Banach manifolds and Banach bundles}

\label{direct_limit_Banach_manifolds_bundles}

\subsection{Direct limits\label{__DirectLimits}}

\begin{definition}
\label{D_DirectedSystem}Let $\left(  I,\leq\right)  $ be a directed set and
let $\mathbb{A}$ be a category. A \textit{directed system} is a pair
$\mathcal{S}=\left\{  \left(  Y_{i},\varepsilon_{i}^{j}\right)  \right\}
_{i\in I,\ j\in I,\ i\leq j}$ where $Y_{i}$ is an object of the category and
$\varepsilon_{i}^{j}:Y_{i}\longrightarrow Y_{j}$ is a morphism
(\textit{bonding map}) where:

\begin{description}
\item[\textbf{(DS 1)}] for all integer $i,\ \varepsilon_{i}^{i}%
=\operatorname{Id}_{Y_{i}}$;

\item[\textbf{(DS 2)}] for all integers $i\leq j\leq k$, $\varepsilon_{j}%
^{k}\circ\varepsilon_{i}^{j}=\varepsilon_{i}^{k}$.
\end{description}
\end{definition}

When $I=\mathbb{N}$ with the usual order relation, countable direct systems
are called \textit{direct sequences}.

\begin{definition}
\label{D_Cone}A cone over $\mathcal{S}$ is a pair $\left\{  \left(
Y,\varepsilon_{i}\right)  \right\}  _{i\in I}$ where $Y\in\operatorname*{ob}%
\mathbb{A}$ and $\varepsilon_{i}:Y_{i}\longrightarrow Y$ is such that
$\varepsilon_{i}^{j}\circ\varepsilon_{i}=\varepsilon_{j}$ whenever $i\leq j$.
\end{definition}

\begin{definition}
\label{D_DirectLimit}A cone $\left\{  \left(  Y,\varepsilon_{i}\right)
\right\}  _{i\in I}$ is a \textit{direct limit} of $\mathcal{S}$ if for every
cone $\left\{  \left(  Z,\theta_{i}\right)  \right\}  _{i\in I}$ over
$\mathcal{S}$ there exists a unique morphism $\psi:Y\longrightarrow Z$ such
that $\psi\circ\varepsilon_{i}=\theta_{i}$. We then write $Y=\underrightarrow
{\lim}\mathcal{S}$ or $Y=\underrightarrow{\lim}Y_{i}$.
\end{definition}


\subsection{Direct limit of ascending sequences of Banach
spaces\label{AscendingSequencesOfSupplementedBanachSpaces}}${}$


Let $\left(  M_{n}\right)  _{n\in\mathbb{N}^{\ast}}$ be a sequence of Banach
spaces such that
\[
\mathbb{M}_{1}\subset\mathbb{M}_{2}\subset\cdots\subset\mathbb{M}_{n}%
\subset\cdots
\]
where $M_{n+1}$ is closed in $M_{n}$. if $\iota_{n}^{n+1}$ is the natural
inclusion of $\mathbb{E}_{n}$ in $\mathbb{E}_{n+1}$ we say that $\{\left(
\mathbb{M}_{n},\iota_{n}^{n+1}\right)  \}_{n\in\mathbb{N}^{\ast}}$ is
\textit{a (strong) ascending sequence of Banach spaces}.\newline

Assume that for any $n\in\mathbb{N}^{\ast}$, $\mathbb{M}_{n}$ is supplemented
in $\mathbb{M}_{n+1}$ and let $\mathbb{M}_{1}^{\prime}$, $\mathbb{M}%
_{2}^{\prime}$, ... be Banach spaces such that:%
\[
\left\{
\begin{array}
[c]{c}%
\mathbb{M}_{1}=\mathbb{M}_{1}^{\prime}\\
\forall n\in\mathbb{N}^{\ast},\ \mathbb{M}_{n+1}\simeq\mathbb{M}_{n}%
\times\mathbb{M}_{n+1}^{\prime}%
\end{array}
\right.
\]

For $i,j\in\mathbb{N}^{\ast},\;i<j$, consider the injections
\[%
\begin{array}
[c]{cccc}%
\iota_{i}^{j}: & \mathbb{M}_{i}\backsimeq\mathbb{M}_{1}^{\prime}\times
\cdots\times\mathbb{M}_{i}^{\prime} & \rightarrow & \mathbb{M}_{j}%
\backsimeq\mathbb{M}_{1}^{\prime}\times\cdots\times\mathbb{M}_{j}^{\prime}\\
& (x_{1}^{\prime},\dots,x_{i}^{\prime}) & \mapsto & (x_{1}^{\prime}%
,\dots,x_{i}^{\prime},0,\dots,0)
\end{array}
\]

$\left\{  (\mathbb{M}_{n},\iota_{n}^{n+1})\right\}  _{n\in\mathbb{N}^{\ast}}$
is called an \textit{ascending sequence of supplemented Banach spaces.}%
\newline In each one of the previous situations, the direct limit
\[
\mathbb{M}=\bigcup\limits_{n\in\mathbb{N}^{\ast}}\mathbb{M}_{n}%
=\underrightarrow{\lim}\mathbb{M}_{n}%
\]
has a structure of convenient space (\cite{KrMi}).


\subsection{Direct limits of ascending sequences of Banach
manifolds\label{DirectLimitsOfAscendingSequencesOfBanachManifolds}}

\begin{definition}
\label{D_AscendingSequenceOfBanachManifolds}\bigskip$\mathcal{M}%
=(M_{n},\varepsilon_{n}^{n+1})_{n\in\mathbb{N}^{\ast}}$ is called an ascending
sequence of Banach manifolds if for any $n\in\mathbb{N}^{\ast}$, $\left(
M_{n},\varepsilon_{n}^{n+1}\right)  $ is a weak submanifold of $M_{n+1}$ where
the model $\mathbb{M}_{n}$ is supplemented in $\mathbb{M}_{n+1}.$
\end{definition}

\begin{proposition}
\label{P_ConditionsDefiningDLCP}Let $\mathcal{M}=\left\{  (M_{n}%
,\varepsilon_{n}^{n+1})\right\}  _{n\in\mathbb{N}^{\ast}}$ be an ascending
sequence of Banach manifolds.\newline Assume that for $x\in M=\underrightarrow
{\lim}M_{n}$, there exists a sequence of charts $\left\{  (U_{n},\phi
_{n})\right\}  _{n\in\mathbb{N}^{\ast}}$ of $M_{n}$, for each $n\in
\mathbb{N}^{\ast}$, such that:

\begin{description}
\item[\textbf{(ASC 1)}] $(U_{n})_{n\in\mathbb{N}^{\ast}}$ is an ascending
sequence of chart domains;

\item[\textbf{(ASC 2)}] $\forall n\in\mathbb{N}^{\ast},\ \phi_{n+1}%
\circ\varepsilon_{n}^{n+1}=\iota_{n}^{n+1}\circ\phi_{n}$.\newline Then
$U=\underrightarrow{\lim}U_{n}$ is an open set of $M$ endowed with the
$\mathrm{DL}$-topology and $\phi=\underrightarrow{\lim}\phi_{n}$ is a well
defined map from $U$ to $\mathbb{M}=\underrightarrow{\lim}\mathbb{M}_{n}$.
\newline Moreover, $\phi$ is a continuous homeomorphism from $U$ onto the open
set $\phi(U)$ of $\mathbb{M}$.
\end{description}
\end{proposition}

\begin{definition}
\label{DLChartProperty} We say that an ascending sequence $\mathcal{M}%
=\left\{  (M_{n},\varepsilon_{n}^{n+1})\right\}  _{n\in\mathbb{N}^{\ast}}$ of
Banach manifolds has the direct limit chart property \emph{\textbf{(DLCP)}} if
$(M_{n})_{n\in\mathbb{N}^{\ast}}$ satisfies \emph{\textbf{(ASC 1) }}$et$
\emph{\textbf{(ASC 2)}}.
\end{definition}

We then have the following result proved in \cite{CaPe}.

\begin{theorem}
\label{T_LBManifold} Let $\mathcal{M}=\left\{  (M_{n},\varepsilon_{n}%
^{n+1})\right\}  _{n\in\mathbb{N}^{\ast}}$ be an ascending sequence of Banach
manifolds where $M_{n}$ is modelled on the Banach spaces $\mathbb{M}_{n}$.
Assume that $(M_{n})_{n\in\mathbb{N}^{\ast}}$ has the direct limit chart
property at each point $x\in M=\underrightarrow{\lim}M_{n}$.\newline Then
there is a unique not necessarly Haussdorf convenient manifold structure on
$M=\underrightarrow{\lim}M_{n}$ modelled on the convenient space
$\mathbb{M}=\underrightarrow{\lim}\mathbb{M}_{n}$ endowed with the
$\mathrm{DL}$-topology. \newline
\end{theorem}

Moreover, if each $M_{n}$ is paracompact, then $M=\underrightarrow{\lim}M_{n}$
is provided with a Hausdorff convenient manifold structure. Recall that if
$\mathcal{M}=\left\{  (M_{n},\varepsilon_{n}^{n+1})\right\}  _{n\in
\mathbb{N}^{\ast}}$ is an ascending sequence of Banach manifolds where $M_{n}$
is modelled on the Banach spaces $\mathbb{M}_{n}$ supplemented in
$\mathbb{M}_{n+1}$, it has the direct limit chart property at each point of
$M=\underrightarrow{\lim}M_{n}$.\newline


\subsection{Direct limits of Banach vector
bundles\label{DirectLimitsOfBanachVectorBundles}}


\begin{definition}
\label{D_AscendingSequenceBanachVectorBundles} A sequence $\mathcal{E}%
=\left\{  \left(  E_{n},\pi_{n},M_{n}\right)  \right\}  _{n\in\mathbb{N}%
^{\ast}}$ of Banach vector bundles is called a strong ascending sequence of
Banach vector bundles if the following assumptions are satisfied:

\begin{description}
\item[\textbf{(ASBVB 1)}] $\mathcal{M}=(M_{n})_{n\in\mathbb{N}^{\ast}}$ is an
ascending sequence of Banach $C^{\infty}$-manifolds where each $M_{n}$ is
modelled on $\mathbb{M}_{n}$ supplemented in $\mathbb{M}_{n+1}$ and the
inclusion $\varepsilon_{n}^{n+1}:M_{n}\longrightarrow M_{n+1}$ is a
$C^{\infty}$ injective map such that $(M_{n},\varepsilon_{n}^{n+1})$ is a weak
submanifold of $M_{n+1}$;

\item[\textbf{(ASBVB 2)}] The sequence $(E_{n})_{n\in\mathbb{N}^{\ast}}$ is an
ascending sequence such that the sequence of typical fibers $\left(
\mathbb{E}_{n}\right)  _{n\in\mathbb{N}^{\ast}}$ of $(E_{n})_{n\in
\mathbb{N}^{\ast}}$ is an ascending sequence of Banach spaces and
$\mathbb{E}_{n}$ is a supplemented Banach subspace of $\mathbb{E}_{n+1}$;

\item[\textbf{(ASBVB 3)}] For each $n\in\mathbb{N}^{\ast}$, $\pi_{n+1}%
\circ\lambda_{n}^{n+1}=\varepsilon_{n}^{n+1}\circ\pi_{n}$ where $\lambda
_{n}^{n+1}:E_{n}\longrightarrow E_{n+1}$ is the natural inclusion;

\item[\textbf{(ASBVB 4)}] Any $x\in M=\underrightarrow{\lim}M_{n}$ has the
direct limit chart property {\textbf{(DLCP) }} for $(U=\underrightarrow{\lim
}U_{n},\phi=\underrightarrow{\lim}\phi_{n})$;

\item[\textbf{(ASBVB 5)}] For each $n\in\mathbb{N}^{\ast}$, there exists a
trivialization $\Psi_{n}:\left(  \pi_{n}\right)  ^{-1}\left(  U_{n}\right)
\longrightarrow U_{n}\times\mathbb{E}_{n}$ such that, for any $i\leq j$, the
following diagram is commutative:
\end{description}

\[%
\begin{array}
[c]{ccc}%
\left(  \pi_{i}\right)  ^{-1}\left(  U_{i}\right)  & \underrightarrow
{\lambda_{i}^{j}} & \left(  \pi_{j}\right)  ^{-1}\left(  U_{j}\right) \\
\Psi_{i}\downarrow &  & \downarrow\Psi_{j}\\
U_{i}\times\mathbb{E}_{i} & \underrightarrow{\varepsilon_{i}^{j}\times
\iota_{i}^{j}} & U_{j}\times\mathbb{E}_{j}.
\end{array}
\]

\end{definition}

For example, the sequence $\left\{  \left(  TM_{n},\pi_{n},M_{n}\right)
\right\}  _{n\in\mathbb{N}^{\ast}}$ is a strong ascending sequence of Banach
vector bundles whenever $(M_{n})_{n\in\mathbb{N}^{\ast}}$ is an ascending
sequence which has the direct limit chart property at each point of $x\in
M=\underrightarrow{\lim}M_{n}$ whose model $\mathbb{M}_{n}$ is supplemented in
$\mathbb{M}_{n+1}$.

We then have the following result given in \cite{CaPe}.

\begin{proposition}
\label{P_StructureOnDirectLimitLinearBundles}

Let $\left(  E_{n},\pi_{n},M_{n}\right)  _{n\in\mathbb{N}^{\ast}}$ be a strong
ascending sequence of Banach vector bundles. Then we have:

\begin{enumerate}
\item $\underrightarrow{\lim}E_{n}$ has a structure of not necessarly
Hausdorff convenient manifold modelled on the LB-space $\underrightarrow{\lim
}\mathbb{M}_{n}\times\underrightarrow{\lim}\mathbb{E}_{n}$ which has a
Hausdorff convenient structure if and only if $M$ is Hausdorff.

\item $\left(  \underrightarrow{\lim}E_{n},\underrightarrow{\lim}\pi
_{n},\underrightarrow{\lim}M_{n}\right)  $ can be endowed with a structure of
convenient vector bundle whose typical fibre is $\underrightarrow{\lim
}\mathbb{\mathbb{E}}_{n}$.
\end{enumerate}
\end{proposition}

\begin{corollary}
\label{C_DirectLimitTangentBundle} Consider the sequence $\left\{  \left(
TM_{n},\pi_{n},M_{n}\right)  \right\}  _{n\in\mathbb{N}^{\ast}}$ associated to
an ascending sequence $(M_{n})_{n\in\mathbb{N}^{\ast}}$ which has the direct
limit chart property at each point of $x\in M=\underrightarrow{\lim}M_{n}$,
where $M_{n}$ is modelled on $\mathbb{M}_{n}$. Then $\underrightarrow{\lim
}TM_{n}$ has a convenient vector bundle structure whose typical fibre is
$\underrightarrow{\lim}\mathbb{\mathbb{M}}_{n}$ over $\underrightarrow{\lim
}M_{n}$.
\end{corollary}


\section{Problem of existence of a solution of a direct limit ODE's}

\label{solODE}

\subsection{Sufficient conditions of existence of solutions for direct limit
of ODEs}${}$

\label{sufsolODE}

Let $\left(  \mathbb{E}_{n}\right)  _{n\in\mathbb{N}^{\ast}}$ be an ascending
sequence of Banach spaces and set $\mathbb{E}=\underrightarrow{\lim}%
\mathbb{E}_{n}$. If $\mathbb{E}_{n}$ is supplemented in $\mathbb{E}_{n+1}$,
for each $n\in\mathbb{N}^{\ast}$, we can choose a supplement linear subspace
$\mathbb{E}_{n}^{\prime}$ in $\mathbb{E}_{n+1}$. We can provide the sequence
$\left(  \mathbb{E}_{n}\right)  _{n\in\mathbb{N}^{\ast}}$ with norms
$||\;||_{n}$ given by induction in the following way:

$\bullet\;$ We choose any norm $||\;||_{1}$ on $\mathbb{E}_{1}$;

$\bullet\;$ Assume that we have built a norm $||\;||_{n}$ on $\mathbb{E}_{n},$
we then choose any norm $||\;||_{n}^{\prime}$ on $\mathbb{E}_{n}^{\prime}$ and
we provide $\mathbb{E}_{n+1}$ with the norm $||\;||_{n+1}=||\;||_{n}%
+||\;||_{n}^{\prime}$.\newline

Since by construction, we have an increasing sequence of norms $\left(
\left\Vert \ \right\Vert _{n}\right)  _{n\in\mathbb{N}^{\ast}}$ such that the
restriction of $||\;||_{n+1}$ to $\mathbb{E}_{n}$ is $||\;||_{n}$, then
$||\;||=\underrightarrow{\lim}||\;||_{n}$ is well defined and is a continuous
map on $\mathbb{E}$. More generally we introduce:

\begin{definition}
\label{coherentnorm} We will say that a sequence of increasing $\left(
\left\Vert \ \right\Vert _{n}\right)  _{n\in\mathbb{N}^{\ast}}$is coherent if
the restriction of $||\;||_{n+1}$ to $\mathbb{E}_{n}$ is $||\;||_{n}$ for all
$n\in\mathbb{N}^{\ast}$.
\end{definition}

\begin{proposition}
\label{norm|E} Let $\left(  \left\Vert \ \right\Vert _{n}\right)
_{n\in\mathbb{N}^{\ast}}$ be a sequence of coherent norms. Then
$||\;||=\underrightarrow{\lim}||\;||_{n}$. Moreover, the topology defined by
$||\;||$ is coarser than the $c^{\infty}$-topology of $\mathbb{E}$.
\end{proposition}

\begin{proof}
At first note that for any $x\in\mathbb{E}$, from the coherence of the
sequence of norms, we have $||x||_{n}=||x||_{n}^{\prime}$ if $x\in
\mathbb{E}_{n}\subset\mathbb{E}_{n^{\prime}}$; So $||x||$ is well defined ant
it follows that $||x||=\underrightarrow{\lim}||x||_{n}$. To prove that
$||\;||$ is a norm, from its definition, we only have to show the triangular
inequality because the other properties are obvious from the previous
argument. Fix some $x$ and $y$ in $\mathbb{E}$. Let $n_{0}$ be the smallest
integer $n$ such that $x$ and $y$ belong to $E_{n}$. Therefore we have
\[
\left\Vert x+y\right\Vert =\left\Vert x+y\right\Vert _{n_{0}}\leq\left\Vert
x\right\Vert _{n_{0}}+\left\Vert y\right\Vert _{n_{0}}=\left\Vert x\right\Vert
+\left\Vert y\right\Vert .
\]
Consider the open balls $B(a,r)=\{x\in\mathbb{E}\;:||x-a||<r\}$ and set
$B_{n}(a,r)=B(a,r)\cap E_{n}$. From the coherence of the sequence of norms, we
have $B_{n}(a,r)=\{x\in\mathbb{E}_{n}\;:||x-a||_{n}<r\}$ and so $B_{n}(a,r)$
is open in $\mathbb{E}_{n}$. It follows that $B(a,r)$ is an open set of the
direct limit topology. This implies that the topology defined by $||\;||$ is
coarser than the direct limit topology. But the direct limit topology is the
same as the the $c^{\infty}$-topology of $\mathbb{E}$ (\textit{cf.}
\cite{CaPe}, Proposition 20).\\
\end{proof}

Assume that $\left(  \left\Vert \ \right\Vert _{n}\right)  _{n\in
\mathbb{N}^{\ast}}$is a sequence of coherent norms on $\mathbb{E}_{n}$. We
denote by $\mathcal{L}(\mathbb{E}_{n})$ the set of continuous linear
endomorphisms of $\mathbb{E}_{n}$ provided by the norm $||\;||_{n}%
^{\mathrm{op}}$ associated to $||\;||_{n}$. If $\left(  T_{n}\right)
_{n\in\mathbb{N}^{\ast}}$ is a sequence of endomorphisms of $\mathbb{E}_{n}$,
such that the restriction of $T_{n+1}$ to $\mathbb{E}_{n}$ is $T_{n}$, we have

$||T_{n+1}||_{n+1}^{\mathrm{op}}\leq||T_{n}||_{n}^{\mathrm{op}}$ and
$||(T_{n+1})_{|E_{n}}||_{n+1}^{\mathrm{op}}=||T_{n}||_{n}^{\mathrm{op}}%
$.\newline

\textit{From now, and in this section, we assume that the ascending sequence
}$\left(  \mathbb{E}_{n}\right)  _{n\in\mathbb{N}^{\ast}}$ \textit{of Banach
spaces is provided with a coherent sequence }$\left(  \left\Vert \ \right\Vert
_{n}\right)  _{n\in\mathbb{N}^{\ast}}$\textit{ of norms}.\newline

Given a compact interval $I$, we fix the map 
\[%
\begin{array}
[c]{cccc}%
f=\underrightarrow{\lim}f_{n}: & I\times U & \longrightarrow & E\\
& \left(  t,x\right)  & \longmapsto & \underrightarrow{\lim}f_{n}\left(
t,x_{n}\right)
\end{array}
\]
For any $n>0$, we consider the differential equation in the Banach space
$E_{n}$~:
\begin{equation}
x_{n}^{\prime}=f_{n}\left(  t,x_{n}\right) \label{Eq_ED-EspaceBanach-En}
\end{equation}

Now fix some $a=\underrightarrow{\lim}a_{n}\in U$. Fix some $t_{0}\in I$ and
let $T>0$ be such that $J=]t_{0}-T,t_{0}+T[\subset I$. If $n_{0}$ is the
smallest integer such that $a\in\mathbb{E}_{n}$, then $a_{n}=a_{n_{0}}$ for
$n\geq n_{0}$. For each $n\geq n_{0}$ we consider a closed ball $\bar{B}_{n}(
a_{n}, r_{n})\subset U_{n}$ centred at $a_{n}$ (relative to the norm
$||\;||_{n}$). \newline

Assume that we have the following assumption:\newline

\begin{description}
\item[\textbf{(A}$_{n}$\textbf{)}] There exists a constant $K_{n}>0$ such that:

\begin{description}
\item $\forall(t,x_{n})\in I\times U_{n},\ ||f_{n}(t,x_{n})||_{n}\leq K_{n}$;

\item $\forall(t,x_{n})\in I\times U_{n},\ ||D_{2}f_{n}(t,x_{n})||_{n}^{\mathrm{op}}\leq K_{n}$\\
where $||\;||_{n}^{\mathrm{op}}$ is the operator norm on the set $\mathcal{L}(E_{n})$ of endomorphisms of $E_{n}$ and
$D_{2}$ is the differential relative to the second variable. \newline
\end{description}
\end{description}

\noindent From the classical Theorem of existence of solution of an ordinary
differential equation (\textit{cf.} \cite{Car}, Corollary 1.7.2 for instance),
we can conclude that there exists $0<\tau<\min\{r_{n}/K_{n},T\}$ such that
each differential equations (\ref{Eq_ED-EspaceBanach-En}) has a unique
solution $\gamma_{n}^{x_{n}}$ defined on $J=\left[  t_{0}-\tau,t_{0}%
+\tau\right]  \subset I$ with initial condition $\gamma_{n}^{x_{n}}%
(t_{0})=x_{n}$, for any $x_{n}$ in the open ball $B_{n}(a_{n},r_{n}-\tau
K_{n})$ (relatively to the norm $||\;||_{n}$).\newline

\noindent

Assume that we have following complementary assumptions: \newline

\begin{description}
\item[\textbf{(B)}] The sequence $(r_{n}/K_{n})$ is lower bounded\textit{. }
\end{description}

Then we can choose $\tau$ independent of $n.$

\begin{description}
\item[\textbf{(C)}] The sequence $(r_{n}-\tau K_{n})$ is lower
bounded.
\end{description}

\begin{theorem}
\label{existsol}(existence of solutions) With the previous notations, under
the assumptions \textbf{(}$A_{n}$\textbf{)} for all $n\geq n_{0},$
\textbf{(B)} and \textbf{(C)}, the sequence of open neighbourhoods of $a_{n}$
\[
\left(  V_{n}=B_{n}(a_{n},r_{n}-\tau K_{n})\cap B_{n}(a_{n_{0}},r_{n_{0}}-\tau
K_{n_{0}})\right)
\]
of open neighbourhoods of $a_{n}$ is an ascending sequence of open sets and we
have:\newline

for each $x=\underrightarrow{\lim}x_{n}\in V=\underrightarrow{\lim}V_{n}$,
there exists a unique solution $\gamma^{x}:[t_{0}-\tau,t_{0}+\tau
]\rightarrow\mathbb{E}$ of the ODE
\begin{equation}
x^{\prime}=f\left(  t,x\right)  \label{Eq_ED-EspaceBanach-E}
\end{equation}

with initial condition $\gamma^{x}(t_{0})=x$.
\end{theorem}

\begin{remark}
\label{existsolpt} \normalfont We have  another simpler criterion of existence of an
integral curve for the ODE (\ref{Eq_ED-EspaceBanach-E}). Assume that for some
$a\in\mathbb{E}$, for each $n$, we have an integral curve $\gamma_{n}%
:[t_{0}-\tau,t_{0}+\tau]\rightarrow\mathbb{E}_{n}$ which is a solution of
(\ref{Eq_ED-EspaceBanach-En}) with initial conditions $\gamma_{n}(t_{0}%
)=a_{n}$, for all $n\geq n_{0}$. The assumption that the restriction of
$f_{n+1}$ to $I\times\mathbb{E}_{n}$ is $f_{n}$, implies that $\gamma
=\underrightarrow{\lim}\gamma_{n}$ is well defined on $[t_{0}-\tau,t_{0}%
+\tau]$ and is an integral curve of (\ref{Eq_ED-EspaceBanach-E}) with initial
condition $\gamma(t_{0})=a$.
\end{remark}


\subsection{Comments and Examples}${}$

\label{comments}

The validity of the assumptions \textbf{(}$\mathbf{A}_{n}$\textbf{)} for all
$n\geq n_{0},$ \textbf{(B)} and \textbf{(C)} suggests the following comments:

\begin{comments}
\label{comment}${}$\normalfont

\begin{enumerate}
\item[1.] According to the previous notations, we set:
\[
K_{n}^{0}(a)=\sup_{t\in I}||f_{n}(t,a)||_{n}\;\;\text{ and }K_{n}^{1}%
(a)=\sup_{t\in I}||D_{2}f_{n}(t,a)||_{n}^{\mathrm{op}}%
\]
Assume that we have

\begin{enumerate}
\item[\textbf{(A}$_{n}^{\prime}$\textbf{)}] $K=\sup_{n\in\mathbb{N}^{\ast}%
}(K_{n}^{0}(a),K_{n}^{1}(a))<\infty$, and there exists an ascending sequence
$\left(  U_{n}\right)  $ of neighbourhoods of $a_{n}$ such that \textbf{(}%
$\mathbf{A}_{n}$\textbf{)} is true for $K_{n}=2\sup(K_{n}^{0}(a),K_{n}%
^{1}(a))$;

\item[\textbf{(B')}] $r=\inf_{n\in\mathbb{N}^{\ast}}(r_{n})>0$.
\end{enumerate}

Then, for $0<\tau\leq\inf(r/2K,T)$, all the conditions \textbf{(}%
$\mathbf{A}_{n}$\textbf{)}, \textbf{(B)} and\textbf{ (C)} are satisfied, and
so we can apply Theorem \ref{existsol} (\textit{cf.} Example \ref{exexistsol}-3.)

\item[2.] Note that for each $n$, the assumption \textbf{(}$\mathbf{A}_{n}%
$\textbf{) }is always satisfied at least on some neighbourhood $W_{n}$ of
$a_{n}\in U_{n}$, but since the restriction of $f_{n+1}$ to $U_{n}$ is $f_{n}$
then $K_{n+1}\leq K_{n}$ and $W_{n+1}$ can eventually not contain $W_{n+1}$.
So even after reduction of $U_{n}$ to $W_{n}\cap W_{n_{0}}$, the sequence
$W_{n}\cap W_{n_{0}}$ can be not an ascending sequence of open sets and again
we can have $$\bigcap\limits_{n\geqslant n_{0}}W_{n}\cap W_{n_{0}}=\{a\}.$$ For
instance, under the assumption \textbf{(}$\mathbf{A}_{n}^{\prime}$\textbf{)},
since each $f_{n}$ is a Lipschitz map and the restriction of $f_{n+1}$ to
$U_{n}$ is $f_{n}$, the sequence of open sets
\[
W_{n}=\{x_{n}\in U_{n}\;:\forall t\in I,\;||f_{n}(t,x_{n})||_{n}\leq K\;\}
\]
will be a decreasing sequence of open sets and so the previous situation is
generally what happens (\textit{cf.} Example \ref{exexistsol}-4).

\item[3.] However, assume that we can choose an ascending sequence $\left(
U_{n}\right)  $ such that \textbf{(}$\mathbf{A}_{n}$\textbf{)} is true for all
$n\geq n_{0}$. Since $\left(  K_{n}\right)  $ is an increasing sequence, the
assumption \textbf{(B)} is a condition on the comparison between the growth of
$K_{n}$ and of the diameter of $U_{n}$ and there is no general reason for
which \textbf{(B)} is satisfied (\textit{cf.} Example \ref{exexistsol}-4).

\item[4.] Assume that \textbf{(}$\mathbf{A}_{n}$\textbf{) }and \textbf{(B)}
are satisfied, it is clear that there can exist some opposition between
\textbf{(B)} and \textbf{(C)}.
\end{enumerate}

\textit{\textbf{In conclusion,} according to the fact that we impose the
existence of a sequence of coherent norms, these conditions are too much
strong, the theorem \ref{existsol} seems to be of weak interest \textbf{in the
general case}}.\newline
\end{comments}

\begin{examples}\normalfont
\label{exexistsol} ${}$In all the following examples, unless otherwise
specified, we assume that each ascending sequence $\left(  \mathbb{E}%
_{n}\right)  $ of Banach spaces is provided with a sequence $\left(
\left\Vert \ \right\Vert _{n}\right)  $ of coherent norms and we will use the
previous notations.

\begin{enumerate}
\item[1.] Assume that $f_{n}$ is a linear map in the second variable. Since
the interval $I$ is compact, the assumption \textbf{(}$\mathbf{A}_{n}%
$\textbf{)} is true for each $n\geq n_{0}$ and for any neighbourhood $U_{n}$
of $a$. Note that $K_{n}^{1}(a)$ does not depend on $a$ and on $U_{n}%
={B}(a_{n},R_{n})$ we have
\[
\forall(t,x_{n})\in I\times U_{n},\ ||f_{n}(t,x_{n})||_{n}\leq K_{n}^{1}%
R_{n}\;
\]
If the sequence $(K_{n})$ is bounded, this implies that we can always find a
sequence $\left(  R_{n}\right)  $ such that \textbf{(}$\mathbf{A}_{n}%
$\textbf{), (B) }and \textbf{(C)} are satisfied and Theorem \ref{existsol}
can be applied. Note that, since  for each $n$, we have a linear ODE, we can
directly show the existence of solutions for the ODE
(\ref{Eq_ED-EspaceBanach-E}) simply by using the general result about the
existence of solutions of a linear ODE in a Banach space and Remark
\ref{existsolpt}. Of course, the first description must be considered as an
illustration of the Theorem \ref{existsol}. This also justifies the
conclusions at the end of Comments \ref{comment}.

\item[2.] Assume that $f_{n}=f_{N}$, for all $n>N$. We can have this situation
when $\mathbb{E}_{n}=\mathbb{E}_{N}$ for $n\geq N$ or also when $\mathbb{E}%
_{n}$ is supplemented in $\mathbb{E}_{n+1}$, for $n\geq N$. Then again in this
case, clearly all the assumption \textbf{(}$\mathbf{A}_{n}$\textbf{), (B) }and
\textbf{(C)} are satisfied but, of course, we can get the result directly
without using such conditions.

\item[3.] We provide $\mathbb{R}^{n}$ with the norm $||(x_{1},\dots
,x_{n})||_{n}=\sum\limits_{i=1}^{n}|x_{i}|$. Then $(||\;||)_n$ is a coherent sequence of
norms. Note that if $x\in\mathbb{R}^{\infty}=\underrightarrow{\lim}%
\mathbb{R}^{n}$, there exists an integer $n$ such that $x$ belongs to
$\mathbb{R}^{n}$ and if $n_{0}$ is the smallest of such integers, we have
$||x||=||(x_{1},\dots,x_{n_{0}})||_{n_{0}}$. \newline 
Now, on $\mathbb{R}$, we consider the function $\theta:u\mapsto u^{2}+1$. Then the family of solutions
of the differential equation $u^{\prime}=\theta(u)$ is given by $u(t)=\tan
(t-C)$ for any constant $C$. For $t\in\lbrack-1,1]$ and $y_{i}\in ]-\frac{\pi}{2},\frac{\pi}{2}[$, we set
\[
\varphi_{i}(t,y_{i})=\frac{t}{i^{2}}(y_{i}^{2}+1).
\]
Let $U_{n}$ be the open set $(\left]  -\dfrac{\pi}{2},\dfrac{\pi}{2}\right[
)^{n}\subset\mathbb{R}^{n}$. On $U_{n}$, we consider the differential equation
$x_{n}^{\prime}=f_{n}(t,x_{n})$ where $f_{n}(t,x_{n})=(\varphi_{1}%
(t,y_{1}),\dots,\varphi_{n}(t,y_{n}))$ and $x_{n}=(y_{1},\dots,y_{n})$. Then,
on $U_{n}$, we have%
\[
\forall(t,x_{n})\in\lbrack-1,1]\times U_{n},\ ||f_{n}(t,x_{n})||_{n}\leq\left(
\frac{\pi^{2}}{4}+1\right)  \left(  \sum\limits_{i=1}^{n}\dfrac{1}{i^{2}%
}\right)
\]
\[
\forall(t,x_{n})\in\lbrack-1,1]\times U_{n},\ ||D_{2}f_{n}(t,x_{n})||_{n}%
^{\mathrm{op}}\leq\left(  \frac{\pi^{2}}{4}+1\right)  \left(  \sum_{i=1}^{n}\frac
{1}{i^{2}}\right)
\]

So the assumption \textbf{(}$\mathbf{A}_{n}$\textbf{)} is satisfied for all
$n>0$. Note that we have
\[
U=\underrightarrow{\lim}U_{n}=(]-\frac{\pi}{2},\frac{\pi}{2}[)^{\infty}%
\subset\mathbb{R}^{\infty}.
\]

If $K_{n}=\left(  \dfrac{\pi^{2}}{4}+1\right)  \left(  \sum\limits_{i=1}%
^{n}\dfrac{1}{i^{2}}\right)  $ then we have $K_{n}\leq\left(  \dfrac{\pi^{2}%
}{4}+1\right)  \times\dfrac{\pi^{2}}{6}$. Now the closed ball $\bar{B}%
_{n}(0,\dfrac{3}{2}))$ is contained in $U_{n}$ for each $n>0$ and we have
\[
\frac{3}{2K_{n}}\geq\left(  \dfrac{\pi^{2}}{4}+1\right)  \times\dfrac{\pi^{2}%
}{6}>1
\]
Thus the assumption \textbf{(B)} is satisfied.\newline

We set $K=\left(  \dfrac{\pi^{2}}{4}+1\right)  \times\dfrac{\pi^{2}}{6}>0$ and
we choose $\tau=\dfrac{1}{K}$. Then the assumption \textbf{(C)} is satisfied.
Therefore, we can apply Theorem \ref{existsol} and its conclusion is valid on
the ball $B(0,\dfrac{1}{2})$.

\item[4.] With the same notations as in the previous example we take
\[
\varphi_{i}(t,y_{i})=\frac{t}{i}(y_{i}^{2}+1).
\]
If $f_{n}(t,x_{n})=(\varphi_{1}(t,y_{1}),\dots,\varphi_{n}(t,y_{n}))$, then in
this case on $U_{n}=\left(  \left]  -\dfrac{\pi}{2},\dfrac{\pi}{2}\right[
\right)  ^{n}$, we have%
\[
\forall(t,x_{n})\in\lbrack-1,1]\times U_{n},\ ||f_{n}(t,x_{n})||_{n}%
\leq\left(  \frac{\pi^{2}}{4}+1\right)  \left(  \sum_{i=1}^{n}\frac{1}%
{i}\right)
\]

\[
\forall(t,x_{n})\in\lbrack-1,1]\times U_{n},\ ||D_{2}f_{n}(t,x_{n}%
)||_{n}^{\mathrm{op}}\leq\left(  \frac{\pi^{2}}{4}+1\right)  \left(
\sum_{i=1}^{n}\frac{1}{i}\right)
\]
$\;$

Of course, the condition \textbf{(}$\mathbf{A}_{n}$\textbf{)} will be
satisfied. But, in this situation, for any closed ball $\bar{B}(0,r_{n})$
contained in $U_{n}$ we have $r_{n}\leq\pi/2$. Since the sequence $\left(
\sum\limits_{i=1}^{n}\dfrac{1}{i}\right)  _{n\in\mathbb{N}^{\ast}}$is not
bounded, the condition \textbf{(B)} is not satisfied. Now, we can impose that
for $n\geq N$, we have a constant $K>0$ such that:%
\[
\forall(t,x)\in\lbrack-1,1]\times U_{n},\ ||f_{n}(t,x)||_{n}\leq K
\]%
\[
\forall(t,x)\in\lbrack-1,1]\times U_{n},\ ||D_{2}f_{n}(t,x)||_{n}%
^{\mathrm{op}}\leq K
\]

Then \textbf{(B)} and \textbf{(C)} will be satisfied. But for having such a
relation on $U_{n}$, we must shrink $U_{n}$ for $n>N$ and it is easy to see
that, in this case, we will have $\bigcap\limits_{n\in\mathbb{N}^{\ast}}%
U_{n}=\{0\}$, and so the conditions \textbf{(}$\mathbf{A}_{n}$\textbf{)}
cannot be satisfied.\newline
\end{enumerate}
\end{examples}

\bigskip

\noindent{\small {\textsc{Unit\'e Mixte de Recherche 5127 CNRS, Universit\'e
de Savoie Mont Blanc\newline Laboratoire de Math\'ematiques (LAMA) }%
\newline\textit{Campus Scientifique \newline73370 Le Bourget-du-Lac, France}%
}\newline\noindent{\textsc{E-mail address: }%
\textit{fernand.pelletier@univ-smb.fr}} }

\end{document}